\documentclass[onefignum,onetabnum]{siamart171218}

\usepackage{float}
\usepackage{hhline}
\usepackage{arydshln}
\usepackage{multicol}
\usepackage{amsfonts,amssymb}
\usepackage{gensymb}
\usepackage{graphicx,caption,epstopdf}
\usepackage{algpseudocode}
\usepackage{mathtools}

\ifpdf
\DeclareGraphicsExtensions{.eps,.pdf,.png,.jpg}
\else
\DeclareGraphicsExtensions{.eps}
\fi

\newcommand{\scvx}{\texttt{SCvx }}
\newcommand{\tf}{N}
\newcommand{\ns}{{n_s}}
\newcommand{\nc}{{n_u}}
\newcommand{\nv}{{n_v}}
\newcommand{\nx}{{n_x}}
\newcommand{\nz}{{n_z}}
\newcommand{\vs}{{s'}}
\newcommand{\vsi}{{\vs_i}}
\newcommand{\wnorm}{}
\newcommand{\orderof}{o}
\newcommand{\dsol}{{d^*}}

\newcommand{\bvec}[1]{#1}%\boldsymbol{\mathbf{#1}}}
\newcommand{\real}{\mathbb{R}}
\newcommand{\definedas}{\coloneqq}
\newcommand{\mat}[1]{[#1]}%\begin{bmatrix}#1\end{bmatrix}}

\newcommand{\posvec}{\bvec{p}}
\newcommand{\velvec}{\bvec{v}}

\newcommand{\ncvxeq}{\mathcal{I}_{eq}}
\newcommand{\ncvxineq}{\mathcal{I}_{ineq}}
\newcommand{\cvxineq}{\mathcal{J}_{ineq}}
\newcommand{\Nncvxeq}{e}
\newcommand{\Nncvxineq}{p}
\newcommand{\Ncvxineq}{q}

\newsiamremark{remark}{Remark}
\newsiamthm{assumption}{Assumption}
\crefname{assumption}{Assumption}{Assumptions}
\newsiamthm{condition}{Condition}
\crefname{condition}{Condition}{Conditions}

\Crefname{ALC@unique}{Line}{Lines}

\newsiamthm{problem}{Problem}
\crefname{problem}{Problem}{Problems}
\newcommand{\boxing}[2]{
	\capstartfalse
	\begin{figure}[#1]
		\begin{center}
			\parbox{0.95\textwidth}{
				\rule{0.95\textwidth}{0.4pt}
                \vspace{-4mm}
                #2
                \vspace{-4mm}
                \rule{0.95\textwidth}{0.4pt}
			}
		\end{center}
	\end{figure}
	\capstartfalse
}

\headers{Successive Convexification for Non-convex Optimal Control}{Y. Mao, M. Szmuk, X. Xu, and B. {A\c c\i kme\c se}}

\title{Successive Convexification: A Superlinearly Convergent 
	Algorithm for Non-convex Optimal Control Problems\thanks{Submitted to the editors \today.
	}}

\author{Yuanqi Mao\thanks{University of Washington, Seattle, WA 
		(\email{yqmao@uw.edu}, \email{mszmuk@uw.edu}, \email{xiangru@uw.edu}, \email{behcet@uw.edu}).}
	\and Michael Szmuk\footnotemark[2]
	\and Xiangru Xu\footnotemark[2]
	\and {Beh\c cet\ A\c c\i kme\c se}\footnotemark[2]}

\ifpdf
\hypersetup{
  pdftitle={Successive Convexification for Non-convex Optimal Control},
  pdfauthor={Y. Mao, M. Szmuk, X. Xu, and B. {A\c c\i kme\c se}}
}
\fi

\begin{document}

\maketitle

% REQUIRED
\begin{abstract}
  This paper presents the \scvx algorithm, a successive convexification algorithm designed to solve non-convex constrained optimal control problems with global convergence and superlinear convergence-rate guarantees. 
  The proposed algorithm can handle nonlinear dynamics and non-convex state and control constraints. It solves the original problem to optimality by successively linearizing non-convex dynamics and constraints about the solution of the previous iteration%, and solving the resulting convex subproblem to obtain a solution for the current iteration. Additionally, the algorithm incorporates several safe-guarding techniques into each convex subproblem, employing \textit{virtual controls} and \textit{virtual buffer zones} to avoid artificial infeasibility, and \textit{trust regions} to avoid artificial unboundedness
  . The resulting convex subproblems are numerically tractable, and can be computed quickly and reliably using convex optimization solvers, making the \scvx algorithm well suited for real-time applications.
  Analysis is presented to show that the algorithm converges both globally and superlinearly, guaranteeing i) local optimality recovery: if the converged solution is feasible with respect to the original problem, then it is also a local optimum; ii) strong convergence: if the Kurdyka--{\L}ojasiewicz (KL) inequality holds at the converged solution, then the solution is unique. The superlinear rate of convergence is obtained by exploiting the structure of optimal control problems, showcasing that faster rate of convergence can be achieved by leveraging specific problem properties when compared to generic nonlinear programming methods.
  Numerical simulations are performed for a non-convex quad-rotor motion planning problem, and corresponding results obtained using Sequential Quadratic Programming (SQP) and general purpose Interior Point Method (IPM) solvers are provided for comparison. The results show that the convergence rate of the \scvx algorithm is indeed superlinear, and that \scvx outperforms the other two methods by converging in less number of iterations. %and spending less time per iteration.
\end{abstract}

% REQUIRED

% REQUIRED

\section{Introduction}
The Successive Convexification (\texttt{SCvx}) algorithm was first introduced in \cite{SCvx_cdc16} to solve non-convex optimal control problems with nonlinear system dynamics. This paper extends the \scvx algorithm to include non-convex state and control constraints, further establishes the uniqueness of the converged solution if the Kurdyka--{\L}ojasiewicz (KL) inequality holds, and more importantly, proves that the algorithm converges not only globally, but also superlinearly. This algorithm places minimal requirements on the underlying dynamical system, and applies to a multitude of real-world optimal control problems, such as autonomous landing of reusable rockets \cite{lars2016autonomous}, drone path planning with obstacle avoidance \cite{szmuk2017convexification}, optimal power flow for smart grids \cite{dall2013distributed}, or more generally nonlinear Model Predictive Control (MPC) \cite{mao2019mpc}.

Methods to compute the solution to such non-convex optimal control problems have been the subject of much interest in recent years, and with the proliferation of self-driving cars, reusable rockets, and autonomous drone delivery systems, interest will only intensify. %Finding global solutions to such problems is generally considered NP-hard. Heuristics like simulated annealing \cite{bertsimas2010robust}, or combinatorial approaches like mixed-integer programming \cite{richards2002}, can compute globally optimal solutions for special classes of problems. However, these methods do not scale well with problem size due to their exponential time complexity, and their performance is not always deterministic. Therefore, they are usually not well suited for time and safety critical applications.
These methods can be broadly classified into indirect and direct methods \cite{betts1998survey}. Indirect methods use techniques from classical optimal control theory \cite{berkovitz1974optimal,pontryagin1987mathematical,liberzon2012calculus} to determine the necessary conditions of optimality, which are then solved as a two-point boundary-value-problem. However, the necessary conditions for complex problems, problems with nonlinear dynamics and non-convex state and control constraints, are very difficult to write down, let alone be solved. Direct methods, on the other hand, offer a cleaner solution process and are less affected by the problem complexity. They parameterize the state and/or control signal using a set of basis functions whose coefficients are found via parameter optimization \cite{hull1997,loxton2009optimal}. As such, we will focus on direct methods in this paper.

For the aforementioned parameter optimization problem, often finding a locally optimal solution, or even a dynamically feasible solution, quickly and reliably is preferable to finding the globally optimal solution slowly and non-deterministically. This is particularly true for real-time control systems, where safety, stability, and determinism are typically prioritized over optimality. A common way to find locally optimal solutions is through nonlinear programming, for which there exist a number of readily available off-the-shelf software packages, including the Sequential Quadratic Programming (SQP) solver \texttt{SNOPT} \cite{snopt}, the Interior Point Method (IPM) solver \texttt{IPOPT} \cite{ipopt}, and the SQP-based optimal control interface \texttt{ACADO Toolkit} \cite{acado}. However, the convergence behavior of generic nonlinear programming algorithms is often highly dependent on the initial guess provided to the solver. Furthermore, provided that convergence is achieved, generic nonlinear programming techniques offer few bounds, if any, on the computational effort required to achieve convergence. As a result, these techniques are typically not applicable for real-time autonomous applications.

Convex optimization problems, on the other hand, can be reliably solved in polynomial time to global optimality \cite{nesterov1994interior}. More importantly, recent advances have shown that these problems can be solved in real-time by both generic Second-Order Cone Programming (SOCP) solvers \cite{domahidi2013ecos}, and customized solvers that exploit the specific structure of the problem \cite{mattingley2012,dueri2014automated}. To leverage the power of convex programming in solving non-convex optimal control problems, the non-convex problems must be convexified, {\em transformed into a convex optimization problem}, which is the focus of this paper. Along this line, recent results on a technique known as \textit{lossless convexification} proved that optimal control problems with a certain class of non-convex control constraints can be posed equivalently as relaxed convex optimal control problems \cite{behcet_aut11,pointing2013}. However, lossless convexification cannot handle nonlinear dynamics, non-convex state constraints, or more general non-convex control constraints. Hence, there is a large class of optimal control problems that are not convexifiable at this time.

%Recent research has focused on adapting nonlinear programming methods to optimal control problems while exploiting the underlying structure of the problem. The \scvx algorithm can also be viewed as an effort along these lines.

Sequential Convex Programming (SCP) offers a way to handle more generic nonlinear dynamics and non-convex constraints, and has been applied to quad-rotor motion planning problems \cite{augugliaro2012generation,schulman2014motion}. While SCP usually performs well in practice, no general convergence results have been reported.
%\todo[inline]{How about Ping Lu stuff? Or others?}
%
To properly address the issue of convergence, \cite{SCvx_cdc16} proposed the first version of \scvx algorithm, which utilizes a successive convexification framework to handle nonlinear dynamics with safe-guarding features like \textit{virtual controls} and \textit{trust regions}. More importantly, \cite{SCvx_cdc16} gave a proof of global convergence for the \scvx algorithm. In this paper, we use the SCP and successive convexification synonymously, but prefer the latter name as it suggests a complimentary approach to lossless convexification.

Building on these previous results, this paper presents three main contributions. First, it develops a framework that extends the results from \cite{SCvx_cdc16} to include non-convex state and control constraints. Unlike algorithms specialized to handle specific types of non-convex constraints (for instance, some SCP-type methods \cite{liu2014solving,mao2017successive} are able to handle problems where every constraint function itself is convex, e.g. the union of convex keep-out zones), the framework in this paper is more general, and applies to a larger class of non-convex constraints. Moreover, state and control constraints are handled using the same mechanism used to handle the nonlinear dynamics, thus reducing the implementation complexity of the algorithm.

Second, the convergence result in this paper is stronger than classical numerical optimization schemes in the sense that the whole sequence converges to a single limit point rather than a set of accumulation points. Examples of the latter, i.e. weak convergence, include line-search algorithms \cite[Theorem~2.5.1]{fletcher1987practical}, trust-region methods \cite[Theorem~6.4.6]{conn2000trust} and aforementioned SCP-type methods \cite{liu2014solving,SCvx_cdc16,mao2017successive}. Strong convergence, on the other hand, is a relatively recent development. \cite{absil2005convergence} adapts {\L}ojasiewicz's theorem \cite{lojasiewicz1982trajectoires} from dynamical systems to show strong convergence of iterative numerical optimization with real-analytic cost functions. \cite{bolte2007lojasiewicz,attouch2013convergence} extends the single limit point convergence to non-smooth functions, and for optimization problems that satisfy KL inequality, \cite{attouch2013convergence} provides handy conditions to check for convergence. In this paper, we will demonstrate this type of strong convergence of the \scvx algorithm by verifying those conditions.

Third, under minimal assumptions, we prove that convergence rate of the \scvx algorithm is superlinear, markedly faster than most nonlinear programming methods (e.g. linear convergence for SQP without perfect Hessian approximation \cite[Theorem~3.3]{boggs1995sequential}). In contrast to generic nonlinear programming methods that are largely problem agnostic, our algorithm leverages the structure of optimal control problems to obtain its superlinear rate of convergence. The simulation results in this paper also validate the convergence rate claim. In practice, the \scvx algorithm has already been showcased in multiple agile quad-rotor obstacle avoidance flight demonstrations, where it was executed on an onboard embedded processor in real time \cite{szmuk2017convexification,szmuk2018real}.

The remainder of this paper is organized as follows. In \cref{sec:successive_convexification}, we outline the problem formulation and the \scvx algorithm. In \cref{sec:global}, we provide proofs of global convergence, including both weak and strong convergence. In \cref{sec:superlinear}, we show superlinear convergence rate under control related conditions. In \cref{sec:numerical_results}, we present simulation results of a non-convex quad-rotor motion planning example problem, and compare our results to those produced by SQP or IPM based methods. Lastly, in \cref{sec:conclusion} we end with concluding remarks.

%===============================================================================
\section{Successive Convexification} \label{sec:successive_convexification}

\subsection{Problem Formulation}
A continuous-time optimal control problem has to be discretized before it can be solved using optimization methods on a digital computer \cite{hull1997}. One may use, for instance, Gauss Collocation Method \cite{hager2018convergence} or Pseudospectral Method \cite{garg2010unified} to achieve that. Therefore in this paper, it is sufficient for us to just consider the discrete-time finite-horizon optimal control problem as in~\cref{prob:nonconvex}.
\boxing{h!}{
	\begin{problem} \label{prob:nonconvex}
		Non-Convex Optimal Control Problem \\
		\begin{subequations}
			\begin{equation}
			\underset{u}{\min} \quad C(x,u) \definedas \sum_{i=1}^{\tf}\phi(x_{i},u_{i}), \label{eq:original_cost}
			\end{equation}
			\textnormal{subject to:}
			\begin{align}
			&x_{i+1}=f(x_{i},u_{i})                    & i&=1,2,\ldots,\tf-1, \label{eq:dynamics} \\
			&s(x_{i},u_{i})\leq 0                      & i&=1,2,\ldots,\tf, \label{eq:state_con} \\
			&x_{i}\in X_{i} \subseteq \real^{\nx} & i&=1,2,\ldots,\tf, \label{eq:state_bound} \\
			&u_{i}\in U_{i} \subseteq \real^{\nc} & i&=1,2,\ldots,\tf-1. \label{eq:control_bound}
			\end{align}
		\end{subequations}
	\end{problem}
}

In~\cref{prob:nonconvex}, $ x_{i}$ and $u_{i}$ represent the state and control vectors at the $i^{th}$ discrete time step, $ X_{i}$ and $U_{i} $ are assumed to be convex and compact sets which include the boundary conditions, and $ \tf $ denotes the total number of time steps. For simplicity, and without loss of generality, we assume that $ \tf $ is a fixed integer (hence the final time is fixed; see \cite{szmuk2018successive} for a treatment of free-final-time problems). More concisely, these variables can be written as
\begin{align*}
x &\definedas [ x_1^T, x_2^T, \ldots, x_{\tf}^T   ]^T \,\in X \subseteq \real^{\nx \tf}, \\
u &\definedas [ u_1^T, u_2^T, \ldots, u_{\tf-1}^T ]^T \,\in U \subseteq \real^{\nc(\tf-1)},
\end{align*}
where $ X $ is the Cartesian product of $ X_i $, and $ U $ is the Cartesian product of $ U_i $.
We assume that the function $\phi : \real^\nx \times \real^\nc \rightarrow \real$ in~\cref{eq:original_cost} is convex and continuously differentiable. This is a reasonable assumption for many real-world optimal control problems. For example, the minimum-fuel problem has $ \phi(x_{i},u_{i}) = \|u_{i}\|_2 $.
The system dynamics are represented by~\cref{eq:dynamics}, where $ f : \real^{\nx}\times\real^{\nc} \rightarrow \real^{\nx} $ is assumed to be a continuously differentiable nonlinear function.
The state constraints are represented by~\cref{eq:state_con}, where $ s : \real^{\nx} \times \real^{\nc} \rightarrow \real^{\ns} $ is assumed to be a continuously differentiable non-convex function. Often, lossless convexification can be leveraged to handle the non-convex control constraints (see e.g.~\cite{behcet_aut11,pointing2013}). This should be done whenever possible.

In summary, the non-convexity of the optimal control problem stated above is due to~\cref{eq:dynamics} and~\cref{eq:state_con}.

\subsection{The \scvx Algorithm}

The basic operating principle of the \scvx algorithm involves linearizing the non-convex parts of~\cref{prob:nonconvex} about the solution of the $k^{th}$ iteration. This results in a convex subproblem that is solved to full optimality (which makes this different than standard trust-region based methods),  resulting in a new solution for the $(k+1)^{th}$ iteration. This process is repeated in succession until convergence is achieved. In practice, the \scvx algorithm is simple to initialize, as will be demonstrated in~\cref{sec:numerical_results}. In this paper, we will sometimes use the term \textit{succession} to refer to an iteration of the \scvx algorithm.

In general, the solution to the convex subproblem will not be optimal with respect to the original non-convex problem. To recover optimality, the algorithm must eventually converge to a solution that satisfies the first-order optimality conditions of~\cref{prob:nonconvex}.

%A natural thought would be doing the linearization successively, i.e. at $ (k+1)^{th} $ succession, we linearize the non-convex dynamics and state constraints about the trajectory and the corresponding control computed in the $ k^{th} $ succession. This procedure is repeated  until convergence.

To achieve this, we begin by denoting the solution to the $ k^{th} $ iteration as $ (x^k, u^k) $. At each time step $i$, let 
\begin{align*}
A_i^k &\definedas \dfrac{\partial}{\partial x_i}f(x_i,u_i)\Bigr|_{x_i^k,u_i^k}, &
B_i^k &\definedas \dfrac{\partial}{\partial u_i}f(x_i,u_i)\Bigr|_{x_i^k,u_i^k}, \\
S_i^k &\definedas \dfrac{\partial}{\partial x_i}s(x_i,u_i)\Bigr|_{x_i^k,u_i^k}, &
Q_i^k &\definedas \dfrac{\partial}{\partial u_i}s(x_i,u_i)\Bigr|_{x_i^k,u_i^k},
\end{align*}
and define $d \definedas x-x^k$, $ d_i \definedas x_i-x^k_i$, $w \definedas u-u^k$, and $w_i \definedas u_i-u^k_i$ in terms of the solution to the current iteration, $(x,u)$. At the $i^{th}$ time step, the first-order approximation of \cref{eq:dynamics} and \cref{eq:state_con} about $(x^k_i,u^k_i)$ is given by
\begin{subequations}
	\begin{align}
	&x_{i+1}^k + d_{i+1} = f(x^k_i,u^k_i) + A_i^k d_i + B_i^k w_i, \label{eq:dyn_lin} \\
	&s(x_i^k,u_i^k) + S_i^k d_i + Q_i^k w_i \leq 0, \label{eq:state_con_lin}
	\end{align}
\end{subequations}
which is a linear system with respect to the incremental state and control variables, $ d_i $ and $ w_i $, of the convex subproblem. The linearization procedure affords the benefit of convexity, but introduces two issues that obstruct the convergence process: artificial infeasibility and artificial unboundedness.

The first issue is encountered when a feasible non-convex problem is linearized about a point $(x,u)$ and it results in an infeasible convex subproblem. This undesirable phenomenon is often encountered during the early iterations of the  process, and we refer to it as \textit{artificial infeasibility}.
To mitigate the effects of artificial infeasibility resulting from the linearization of nonlinear dynamics, we augment the linearized dynamics in \cref{eq:dyn_lin} with an unconstrained \textit{virtual control} term, $ v_i \in \real^\nv $
\begin{equation}
x_{i+1}^k + d_{i+1} = f(x_i^k,u_i^k) + A_i^k d_i + B_i^k w_i + E_i^k v_i, \label{eq:dyn_pen}
\end{equation}
%
% Defining $E \definedas [E_1, E_2, \ldots, E_{\tf -1}] $
%
where $E^k_i \in \real^{\nx \times \nv}$ is selected such that $ \textnormal{im}(E^k_i) = \real^{\nx}$, and thus guaranteeing one-step controllability. Since $ v_i $ is left unconstrained, any state in the feasible region of the convex subproblem is reachable in finite time. For example,  for a mass-spring-damper system, the virtual control can be interpreted as a synthetic force  to ensure feasibility with respect to state and control constraints. Consequently, the resulting augmented convex subproblem is no longer vulnerable to artificial infeasibility arising from the linearization of~\cref{eq:dynamics}.

While the virtual control term makes any state reachable in finite time, it does not allow the problem to retain feasibility when the state and control constraints in~\cref{eq:state_con_lin} define an empty feasible set (e.g. consider the union of the state constraints $\mat{1\thinspace,\thinspace 0,\ldots,\thinspace 0}x \geq \epsilon$ and $\mat{1\thinspace,\thinspace 0,\ldots,\thinspace 0}x \leq -\epsilon$, for $\epsilon > 0$).

To mitigate artificial infeasibility caused by the linearization of non-convex state and control constraints, we introduce an unconstrained \textit{virtual buffer zone} term, $ \vsi \in \real^{\ns}_{+} $ (meaning $ \vsi \geq 0 $), to \cref{eq:state_con}:
\begin{equation}
s(x_i^k,u_i^k) + S_i^k d_i + Q_i^k w_i - \vsi \leq 0. \label{eq:state_con_pen}
\end{equation}
$\vsi$ can be understood as a relaxation term that allows the non-convex state and control constraints in~\cref{eq:state_con_lin} to be violated. In an obstacle avoidance example, $\vsi$ can be interpreted as a measure of the obstacle constraint violation necessary to retain feasibility at the $i^{th}$ time step.

%In motion planning problems, adding another layer of ``no entering zone" to the obstacles we want to avoid may serve as an example of such virtual buffer zones.

To ensure that $v_i$ and $\vsi$ are used only when necessary, we define
\begin{align*}
v &\definedas [ v_1^T, v_2^T, \ldots, v_{\tf-1}^T ]^T \in \real^{\nv(\tf-1)} \\
\vs &\definedas [ \vs_1^T, \vs_2^T, \ldots, \vs_{\tf-1}^T ]^T \in \real^{\ns(\tf-1)}_{+},
\end{align*}
and augment the linear cost function with the term $\sum_{i=1}^{\tf-1}{\lambda_i P(E_i v_i,\vsi)},$ where the $ \lambda_i$'s are sufficiently large positive penalty weights, and $ P : \real^\nx \times \real^\ns \rightarrow \real $ is an exact penalty function (which will be more precisely defined later in \cref{eq:sub_cost}). Thus, to obtain the solution for the $ (k+1)^{th} $ iteration, we use the penalized cost given by
\begin{equation}
L^k(d,w) \definedas C(x^k+d,u^k+w) + \sum_{i=1}^{\tf-1}{\lambda_i P(E_i^k v_i,\vsi)}. \label{cost:linear}
\end{equation}
Note that we omitted the dependence of $ v_i $ and $ \vsi $ in $ L^k $ for simplicity. Its corresponding nonlinear penalized cost is given by
\begin{equation}
J(x,u) \definedas C(x,u) +\sum_{i=1}^{\tf-1}{\lambda_i P\big( x_{i+1}-f(x_i,u_i), s(x_i,u_i) \big)}. \label{cost:penalty}
\end{equation}
%
%where $ C(x,u) $ denotes the original cost function $ \sum_{i=1}^{\tf}\phi(x_{i},u_{i}) $.
%
% and \textit{approximation error}, which we . In some cases, the latter can artificially render the problem unbounded.
%
The second adverse effect of linearization is \textit{artificial unboundedness}. This phenomenon occurs when the local properties of the non-convex problems are extrapolated well beyond the neighborhood of the linearization point. For example, consider the following simple non-convex optimization problem
\begin{equation*}
\min \thinspace\thinspace x_2 \thinspace, \quad s.t. \quad x_2 = x_1^2,
\end{equation*}
whose solution is $(x_1^*,x_2^*) = (0,0)$. Linearizing this non-convex problem about any point other than $(x_1^*,x_2^*)$ renders the linearized problem unbounded.

To mitigate the risk of artificial unboundedness, we impose the following trust region constraint to ensure that $u$ does not deviate significantly from the control input obtained in the previous iteration, $u^k$:
$$ \|w\|\wnorm \leq r^k. $$
By selecting $r^k$ appropriately, this constraint, in conjunction with~\cref{eq:dyn_pen}, ensures that $x$ remains sufficiently close to the state vector obtained in the previous iteration, $x^k$, thus keeping the solution within the region where the linearization is accurate.

We now present the final problem formulation and a summary of the \scvx algorithm. At the $(k+1)^{th}$ iteration, we solve the convex optimal control subproblem,~\cref{prob:lin}. For many applications, $ U $ and $ X $ are simple convex sets, e.g. second order cones, in which case \cref{prob:lin} can be readily solved by SOCP solvers (e.g. \cite{dueri2014automated}) in a matter of milliseconds.
\boxing{h}{
	\begin{problem} \label{prob:lin}
		Convex Optimal Control Subproblem \\
		$$ \underset{d, w}{\min}\thinspace\thinspace L^k(d,w), $$
		$\textnormal{subject to:}$ \\
		$$ u^k+w \in U,\quad x^k+d \in X,\quad \|w\|\wnorm \leq r^k. $$ \\
	\end{problem}
}

The \scvx algorithm solves~\cref{prob:nonconvex} according to the steps outlined in \cref{algo:SCvx}. It can be considered as a trust-region-type algorithm, and follows standard trust region radius update rules. However, it also differs from conventional trust region methods in some important ways, as discussed later in \cref{sec:comparison}.
%\todo[inline]{Please read this part! I think "below", which I removed meant Prob 2}

\begin{algorithm} [h]
	\caption{The \scvx Algorithm}
	\label{algo:SCvx}
	\begin{algorithmic}[1]
		\Procedure{\texttt{SCvx}}{$x^{1},u^{1},\lambda,\epsilon_{tol}$}
		\vspace{+2mm}
		\State\textbf{input} Select initial state $ x^{1} \in X $ and control $ u^{1}\in U $. Initialize trust region radius $ r^{1} > 0 $. Select penalty weight $ \lambda > 0 $, and parameters $ 0<\rho_0<\rho_1<\rho_2<1 $, $ r_l>0 $ and $ \alpha>1, \beta>1 $. \label{line:input}
		\While {not converged, i.e. $ \Delta J^k > \epsilon_{tol} $}
		\State \textbf{step 1} At $ (k+1)^{th} $ succession, solve \cref{prob:lin} at $ (x^k,u^k,r^k) $ to get an optimal solution $ (d, w) $.
		\State \textbf{step 2} Compute the \emph{actual} change in the penalty cost \cref{cost:penalty}:
		\begin{equation} \label{eq:actual}
		\Delta J^k=J(x^k,u^k)-J(x^k+d,u^k+w),
		\end{equation}
		and the \emph{predicted} change by the convex cost \cref{cost:linear}:
		\begin{equation} \label{eq:predict}
		\Delta L^k=J(x^k,u^k)-L^k(d,w).
		\end{equation}
		\If {$ \Delta J^k=0 $} \State \textbf{stop}, and \underline{\textbf{return} $ (x^k,u^k) $};
		\Else \State compute the ratio
		\begin{equation} \label{eq:ratio}
		\rho^k \definedas \Delta J^k/\Delta L^k.
		\end{equation}
		\EndIf
		\State \textbf{step 3}
		\If {$ \rho^k < \rho_0 $} \State reject this step, contract the trust region radius, i.e. $ r^k\leftarrow r^k/\alpha $ and go back to \textbf{step 1}; \label{line:reject}
		\Else \State accept this step, i.e. $ x^{k+1}\leftarrow x^k+d $, $ u^{k+1}\leftarrow u^k+w $, and update the trust region radius $ r^{k+1} $ by
		\label{line:accept}
		\begin{equation*}
		r^{k+1}=\begin{cases}
		r^k/\alpha, & \text{if }  \rho^k<\rho_1;\\
		r^k, & \text{if } \rho_1\leq \rho^k<\rho_2;\\
		\beta r^k, & \text{if } \rho_2\leq \rho^k.
		\end{cases}
		\end{equation*}
		\EndIf
		\State  $ r^{k+1}\leftarrow \max\{r^{k+1}, r_l \} $, $ k\leftarrow k+1 $, and go back to \textbf{step 1}.
		\EndWhile
		\State \underline{\textbf{return} $ (x^{k+1},u^{k+1}) $}.
		\vspace{+2mm}
		\EndProcedure
	\end{algorithmic}
\end{algorithm}

In line \ref{line:input}, we initialize the algorithm by using somewhat random initial state $ x^1 $ and control $ u^1 $, which need \textbf{not} to be dynamically feasible. This is important because finding a feasible solution itself is a difficult non-convex problem, and \scvx fortunately does not require that. There are no strict rules to following in choosing parameters. Generally speaking, $ \rho_0 $ is very close to 0, $ \rho_1 $ is slightly greater than 0, and $ \rho_2 $ can be marginally less than 1. $ r_l $ can be slightly greater than 0.

The quality of the linear approximation used in~\cref{prob:lin} can be understood by inspecting the ratio $\rho^k$ in \cref{eq:ratio}, which compares the realized (nonlinear) cost reduction $\Delta J^k$ in \cref{eq:actual}, to the predicted (linear) cost reduction $\Delta L^k$ in \cref{eq:predict} based on the previous (i.e. $k^{th}$) iteration.
In line \ref{line:reject}, when $\rho_k$ is small (i.e. when $\rho^k < \rho_0 \ll 1$), the approximation is considered inaccurate, since the linear cost reduction over-predicts the realized nonlinear cost reduction. Consequently, the solution $(d,w)$ is rejected, the trust region is contracted by a factor of $\alpha < 1$, and the step is repeated. As will be shown in~\cref{sec:global}, the contraction of the trust region ensures that this condition can occur only a finite number of times, thus guaranteeing that the algorithm will not remain in this state indefinitely.

In line \ref{line:accept}, if $\rho^k$ is such that the linearization accuracy is deficient, yet acceptable (i.e. when $\rho_0 \leq \rho^k < \rho_1$), then the solution $(d,w)$ is accepted, but the trust region is contracted. The former is done to avoid unnecessarily discarding the solution that has already been computed, while the latter is done to improve the linearization accuracy of the next succession.

When $\rho^k$ is sufficiently large, yet significantly less than unity (i.e. when $\rho_1 \leq \rho^k < \rho_2$), the linearization is deemed sufficiently accurate. Consequently, the solution $(d,w)$ is accepted, and the trust region size is retained.

Lastly, when $\rho^k$ is close to unity, the linear cost reduction accurately predicts the realized nonlinear cost reduction. Moreover, if $\rho^k$ is greater than unity, then the linear approximation under-predicts the cost reduction, and is thus conservative. These conditions indicate that the linearization is accurate or conservative enough to enlarge the trust region. Hence, when $\rho_k \geq \rho_2$, the solution $(d,w)$ is accepted, and the trust region size is increased by a factor of $\beta > 1$ to allow for larger $w$, and therefore $d$, in the next succession.

%In Step 2 of the \scvx algorithm, the ratio $ \rho^k $ is used as a metric for the quality of linear approximations. The linear approximation is considered accurate when when $ \Delta J^k $ agrees with $ \Delta L^k $, that is, when $ \rho^k $ is close to 1. Hence if $ \rho^k $ is above $ \rho_2 $, which means our linear approximation predicts the cost reduction  well, then we may choose to enlarge the trust region in Step 3, i.e., we  put more ``trust" in our approximation. Otherwise, we may keep the trust region unchanged, or contract its radius if needed. The most unwanted situation is when $ \rho^k $ is negative, or close to zero. The current step will be rejected in this case, and one has to contract the trust region and re-optimize at $ (x^k,u^k) $. As we will see in the next section, however, this undesirable situation may only occur for a finite number of times.

\subsection{Comparison with Standard Trust-Region and SQP Methods} \label{sec:comparison}
One important distinction between \scvx and typical trust-region-type algorithms lies in the subproblem solved at each iteration. Conventional trust-region algorithms usually perform a line search along the Cauchy arc to achieve a sufficient reduction \cite{conn2000trust}. However, in the \scvx algorithm, each convex subproblem is solved to full optimality, thus increasing the number of inner solver iterations (e.g. IPM iterations) at each succession, while decreasing the number of outer \scvx iterations. Qualitatively speaking, the number of successions is reduced by achieving a greater cost reduction at each succession. Thanks to the capabilities of existing IPM algorithms, and due to recent advancements in IPM customization techniques (e.g. \cite{domahidi2013ecos,dueri2014automated}), each convex subproblem can be solved quickly enough to outweigh the negative impacts of solving it to full optimality.

It is also crucial to point out some key differences between the \scvx algorithm and SQP-based methods (e.g. \cite{boggs1995sequential}). SQP-based methods typically make use of second-order information when approximating the Hessian of the cost function (and in some cases, of the constraints). This requires techniques like the Broyden-Fletcher-Goldfarb-Shanno (BFGS) update, which can be computationally expensive. Furthermore, additional conditions must be satisfied in order to ensure that the computed Hessian is positive-definite, and therefore, to guarantee that the resulting subproblem is convex \cite{fletcher1987practical}. For these two reasons, SQP-based methods are not well suited for real-time autonomous applications. On the other hand, the \scvx algorithm relies only on first-order information obtained through linearization. While the first-order approximation may be less accurate than its second-order counterpart, the resulting error is properly regulated by the trust region updating mechanism outlined in~\cref{algo:SCvx}. Furthermore, since the Jacobian matrices can be determined analytically, very little computational effort is expended in setting up each succession. Most importantly, as a consequence of linearization, the resulting subproblems are automatically \textit{guaranteed} to be convex, thus further ensuring the robustness of the convergence process.

%===============================================================================
\section{Global Convergence} \label{sec:global} This section presents two main convergence results. The first one, weak convergence, means that the algorithm generates a set of convergent subsequences, all of which satisfy the first order conditions. This is not convergence in its strict sense due to potential oscillation between several limit points (i.e., accumulation points), but suprisingly most of the convergence claims of nonlinear optimization schemes fall into this category (see \cite{absil2005convergence} and the references therein). It is also relatively easy to prove due to the well-known Bolzano-Weierstrass theorem (see e.g. \cite{rudin1964principles}). In contrast, strong convergence claims the uniqueness of the limit point, which is a much harder argument to make. Fortunately, for certain problem classes, e.g. problems with semi-algebraic functions, \cite{attouch2013convergence} provides conditions that can be checked in advance to attest strong convergence. By verifying those conditions, we will show that, under mild assumptions, the \scvx algorithm is indeed strongly convergent.

\subsection{Weak Convergence}
In this section, we extend the convergence result from \cite{SCvx_cdc16} to facilitate the addition of state and control constraints. Since the state and control variables become indistinguishable once the optimal control problem is implemented as a numerical parameter optimization problem, we perform the following analysis accordingly. In lieu of the state and control variables, $x$ and $u$, we will use $ z \definedas [ x^T, \, u^T ]^T $ as our optimization variable, where $ z \in \real^\nz $ and where $ \nz = \nx \tf + \nc(\tf-1) $. Consequently, the constraints will be rewritten as a set of inequalities expressed in terms of $ z $.
Thus, we have the finite-dimensional non-convex optimization problem in \cref{prob:scvx_original}.

\boxing{!h}{
	\begin{problem} \label{prob:scvx_original}
		Original Non-Convex Problem
		\begin{subequations}
			\begin{align}
			&\underset{z}{\min}       & g_0(z)&,          &  & \label{eq:o_cost} \\
			&\textnormal{subject to:} & g_{i}(z)&=0,      & \forall i&\in\ncvxeq, \label{eq:o_ncvxeq} \\
			&                         & g_{i}(z)&\leq 0,  & \forall i&\in\ncvxineq, \label{eq:o_ncvxineq} \\
			&                         & h_{j}(z)&\leq 0,  & \forall j&\in\cvxineq, \label{eq:o_cvxineq}
			\end{align}
		\end{subequations}
	\end{problem}
}

In \cref{prob:scvx_original}, $ \ncvxeq \definedas \{1,2,\ldots,\Nncvxeq\} $ represents the set of non-convex equality constraint indices, $\ncvxineq \definedas \{\Nncvxeq+1,\ldots,\Nncvxineq\}$, represents the set of non-convex inequality constraint indices, and $\cvxineq \definedas \{1,2,\ldots,\Ncvxineq\}$ represents the set of convex inequality indices. Correspondingly, equations \cref{eq:o_ncvxeq}-\cref{eq:o_cvxineq} represent the nonlinear system dynamics, the non-convex state and control constraints, and the convex state and control constraints, respectively. We assume that $g_i(z)$ and $h_j(z)$ are continuously differentiable for all $i\in \ncvxeq\,\cup\,\ncvxineq$ and $j\in\cvxineq$, respectively. To simplify the analysis that follows, we assume that $ g_0(z) \in C^1 $, but note that $g_0(z)$ can be an element of $C^0$ in practice.

Since we are restricting our analysis to discrete-time systems, $ z $ is a finite dimensional vector. Consequently, the 1-norm used in \cite{SCvx_cdc16} manifests itself as a finite sum of absolute values. Therefore, to incorporate state and control inequality constraints, the exact penalty function $ \gamma(\cdot) $ used in \cite{SCvx_cdc16} can be extended as follows
\begin{align}
J(z) \definedas \; &g_0(z) + \sum_{i\in\ncvxeq}\lambda_i |g_{i}(z)| \notag + \sum_{i\in\ncvxineq}\lambda_i \max \big(0,g_{i}(z)\big) \\
&+ \sum_{j\in\cvxineq}\tau_j \max \big(0,h_{j}(z)\big), \label{eq:pen_cost}
\end{align}
where $ \lambda_i \geq 0$ and $\tau_j \geq 0 $ are scalars, and $ \lambda \definedas [ \lambda_1, \lambda_2, \ldots, \lambda_p ] $ and $ \tau \definedas [ \tau_1, \tau_2, \ldots, \tau_q ] $ represent the penalty weights. To facilitate subsequent proofs, we include the convex constraints, $h_j(z)$, in $ J(z) $. However, in practice, these constraints can be excluded from the cost function and included as explicit convex constraints. Next, we can now express the corresponding penalized problem in \cref{prob:scvx_penalty}.

%\begin{equation}
%\label{eq:scvx_penalty}
%\begin{aligned}
%&\underset{z}{\min} &J(z)\qquad \, & \\
%&\textnormal{subject to:} & h_{j}(z)\leq 0, \quad  & j=1,2,\ldots,q.
%\end{aligned}
%\end{equation}
%Denote the set of $ z $ satisfying the above constraint as
%\begin{equation}
%G=\{z\,|\,h_{j}(z)\leq 0, \quad j=1,2,\ldots,q\}. \label{eq:fea_set}
%\end{equation}
\boxing{!h}{
	\begin{problem} \label{prob:scvx_penalty}
		Non-Convex Penalty Problem
		$$\underset{z}{\min} \;\; J(z), \quad \forall \, z \in \real^\nz.$$
	\end{problem}
}

Note that $ J(z) $ is non-convex, and thus needs to be convexified. According to the \scvx algorithm, at $ (k+1)^{th} $ succession, we linearize $ g_0(z)$, $g_{i}(z) = 0$ for all $i\in\ncvxeq\,\cup\,\ncvxineq$, and $ h_{j}(z)$ for all $j\in\cvxineq$ about $ z^{k} $. This procedure is repeated until convergence, and produces a sequence of convex penalty functions given by
\begin{align}
L^k(d) \definedas \; &g_0(z^k) + \nabla g_0(z^k)^T d + \sum_{i\in\ncvxeq}\lambda_i |g_{i}(z^k) + \nabla g_i(z^k)^T d| \notag \\
&+ \sum_{i\in\ncvxineq}\lambda_i \max \big( 0,g_{i}(z^k)+\nabla g_i(z^k)^T d \big) \notag \\
&+ \sum_{j\in\cvxineq}\tau_j \max \big( 0,h_{j}(z^k) +\nabla h_j(z^k)^T d \big). \label{eq:sub_cost}
\end{align}
Note that $ L^k(d) $ is \textbf{not} exactly linearized $ J(z) $, since the linearization is applied inside the $ |\cdot| $ and $ \max(0, \cdot) $ functions. $ L^k(d) $ is then incorporated into the linearized penalty problem in \cref{prob:scvx_sub}.

\boxing{!h}{
	\begin{problem} \label{prob:scvx_sub}
		Linearized Penalty Problem
		\begin{align*}
		\qquad &\underset{d}{\min} && L^k(d) \qquad & \\
		%&\textnormal{subject to:} & h_{j}(z^k+d)\leq 0, & \quad j=1,2,\ldots,q \\
		\qquad &\textnormal{subject to:} && \|d\| \leq r^k. \qquad &
		\end{align*}
	\end{problem}
}

Again, we remind the reader that the convex constraints are linearized to simplify the notations in the analysis that follows, and that, in practice, they are handled explicitly (in their original form) by the convex optimization solver. The corresponding \emph{actual} cost reduction in \cref{eq:actual} can be rewritten as
\begin{equation}
\Delta J(z,d) = J(z) - J(z+d), \label{eq:DJ}
\end{equation}
while the \emph{predicted} cost reduction in \cref{eq:predict} becomes
\begin{equation}
\Delta L(d) = J(z) - L(d). \label{eq:DL}
\end{equation}
To proceed with convergence analysis, we now introduce some preliminary results and assumptions.
\begin{theorem}[\textbf{Local Optimum}, Theorem~4.1 in \cite{han1979exact}] \label{thm:local_opt}
	Let $ N(\bar{z}) $ denote an open neighborhood of $ \bar{z} $ that contains feasible point of~\cref{prob:scvx_original}. Then, if there exist $ \bar{\lambda}\geq 0$, $\bar{\tau}\geq 0 $, and $ \bar{z}\in \real^\nz $ such that $ J(\bar{z})\leq J(z) $ for all penalty weights $ \lambda\geq \bar{\lambda}$ and $\tau\geq \bar{\tau} $, and for all $ z\in N(\bar{z}) $, then $ \bar{z} $ is a local optimum of~\cref{prob:scvx_original}.
\end{theorem}
%
% We denote the state and control variables associated with $\bar{z}$ by $\bar{x}$ and $\bar{u}$, respectively.
%
\begin{assumption}[\textbf{LICQ}] \label{asup:LICQ}
	Define the following sets of indices corresponding to the active inequality constraints:
%	\begin{align*}
%	\null\qquad I_{ac}(\bar{z}) &\definedas \{\,i\;|\;g_{i}(\bar{z})=0,&i&\in\ncvxineq\} && \subseteq \ncvxineq,\qquad\qquad\null \\
%	\null\qquad J_{ac}(\bar{z}) &\definedas \{\,j\;|\;h_{j}(\bar{z})=0,&j&\in\cvxineq\}  && \subseteq \cvxineq.\qquad\qquad\null
%	\end{align*}
	\begin{align*}
	I_{ac}(\bar{z}) &\definedas \{\,i\;|\;g_{i}(\bar{z})=0, \; i\in\ncvxineq\} \subseteq \ncvxineq,\\
	J_{ac}(\bar{z}) &\definedas \{\,j\;|\;h_{j}(\bar{z})=0, \; j\in\cvxineq\} \subseteq \cvxineq.
	\end{align*}
	Furthermore, define $G_{eq}(\bar{z})$ as a matrix whose columns comprise of $\nabla g_i(\bar{z})$ for all $i\in\ncvxeq$, $G_{ac,g}(\bar{z})$ as a matrix whose columns comprise of $\nabla g_i(\bar{z})$ for all $i\in I_{ac}(\bar{z})$, and $G_{ac,h}(\bar{z})$ as a matrix whose columns comprise of $\nabla h_i(\bar{z})$ for all $j\in J_{ac}(\bar{z})$. Then, the Linear Independence Constraint Qualification (LICQ) is satisfied at $ \bar{z} $ if the columns of the following matrix are linearly independent:
	\begin{equation} \label{eq:G_ac}
	G_{ac}(\bar{z}) \definedas \{\,G_{eq}(\bar{z})\,;\,G_{ac,g}(\bar{z})\,;\,G_{ac,h}(\bar{z})\,\}.
	\end{equation}
\end{assumption}
The following theorem states the first-order necessary conditions (i.e. the KKT conditions) for a point $ \bar{z} $ to be a local optimum of~\cref{prob:scvx_original} (in the same sense as~\cref{thm:local_opt}).

\begin{theorem}[\textbf{Karush--Kuhn--Tucker (KKT)}, Theorem~3.2 in \cite{han1979exact}] \label{thm:kkt}
	If the constraints of~\cref{prob:scvx_original} satisfy the LICQ assumption (\cref{asup:LICQ}), and $ \bar{z} $ is a local optimum of the original problem, \cref{prob:scvx_original}, then there exist Lagrange multipliers $ \bar{\mu}_i$ for all $i\in\ncvxeq$, $ \bar{\mu}_i\geq 0$ for all $i\in\ncvxineq$, and $ \bar{\sigma}_j\geq 0$ for all $j\in\cvxineq$ such that
	\begin{equation}
	\nabla g_0(\bar{z}) + \sum_{i\in \ncvxeq}\bar{\mu}_i \nabla g_i(\bar{z}) + \sum_{i\in I_{ac}(\bar{z})}\bar{\mu}_i \nabla g_i(\bar{z}) + \sum_{j\in J_{ac}(\bar{z})}\bar{\sigma}_j \nabla h_j(\bar{z}) = 0. \label{eq:kkt}
	\end{equation}
\end{theorem}

We refer to a point that satisfies the above conditions as a \textit{KKT point}.

We say a penalty function is \textit{exact} if there exists finite penalty weights $\bar{\mu}_i$ and $\bar{\sigma}_j$ such that~\cref{prob:scvx_original} and~\cref{prob:scvx_penalty} produce equivalent optimality conditions. Since~ \cref{thm:kkt} already gives the optimality condition for~\cref{prob:scvx_original}, we now examine the first-order conditions for~\cref{prob:scvx_penalty}.

For fixed $ \lambda $, the cost $ J(z) $ is not differentiable everywhere because of the non-smoothness of $ |\cdot| $ and $ \max(0,\cdot) $. However, since $ g_i(\cdot) $ and $ h_j(\cdot) $ are both continuously differentiable, $ J(z) $ is locally Lipschitz continuous. Hence we need the following definition:
\begin{definition}[\textbf{GDD}, Definition~1.3 in \cite{clarke1975generalized}] \label{def:gdd}
	If $ J(z) $ is locally Lipschitz continuous, then the Generalized Directional Derivative (GDD) of $ J(z) $ at some $ \bar{z} $ in any direction $ s $ exists, and is defined as follows:
	\begin{equation} \label{eq:gdd}
	dJ(\bar{z},s) \definedas \limsup\limits_{\mathclap{\substack{z\rightarrow \bar{z} \\ \delta\rightarrow 0^+}}} \frac{J(z+\delta s)-J(z)}{\delta}.
	\end{equation}
\end{definition}
Hence, $ dJ $ also satisfies the following implicit relationship by using similar reasoning to that used in~\cite[Lemma~14.2.5]{fletcher1987practical},
\begin{equation} \label{eq:gdd2}
dJ(\bar{z},s)=\max \{\nu^T s \,|\,\nu\in \partial J(\bar{z})\},
\end{equation}
where $ \partial J(\bar{z}) $ is the \textit{generalized differential}, defined as
\begin{equation} \label{eq:gen_diff}
\partial J(\bar{z})=\{\nu\,|\,J(\bar{z}+y)\geq J(\bar{z}) + \nu^T y, \forall \, y\in \real^\nz \}.
\end{equation}
Applying Theorem 1 from~\cite{clarke1976new} with the above definitions, we have
\begin{lemma}[\textbf{Necessary Condition for Local Optimality}] \label{lem:stationary}
	If $ \bar{z} $ is a locally optimal solution of the non-convex penalty problem,~\cref{prob:scvx_penalty}, then $ 0\in \partial J(\bar{z}) $.
\end{lemma}
Next, define the set of stationary points of~\cref{prob:scvx_penalty} as
$$ S \coloneqq \{z\in \real^\nz |\,  0\in \partial J (z)\}. $$
Then,~\cref{lem:stationary} states that if $ \bar{z} $ solves~\cref{prob:scvx_penalty}, then $ \bar{z}\in S $.

We now present the \textit{exactness} result. Note that this theorem is fairly well-known, and its proof can be found in, for example, Theorem~4.4 and~4.8 of~\cite{han1979exact}.
\begin{theorem}[\textbf{Exactness of the Penalty Function}] \label{thm:exactness}
	If $\bar{z}$ is a KKT point of the original ~\cref{prob:scvx_original} with multipliers $\bar{\mu}_i$ for all $i\in\ncvxeq\cup\ncvxineq$ and $ \bar{\sigma}_j$ for all $j\in\cvxineq$, and if the penalty weights $\lambda$ and $\tau$ satisfy
	\begin{align*}
	\lambda_i &> |\bar{\mu}_i|, \quad \forall \; i\in\ncvxeq\cup\ncvxineq, \\
	\tau_j &> |\bar{\sigma}_j|, \quad \forall \; j\in\cvxineq,
	\end{align*}
	then $\bar{z} \in S$, and by definition $\bar{z}$ satisfies the stationarity condition $0\in \partial J(\bar{z})$ for the penalty~\cref{prob:scvx_penalty}. \\
	Conversely, if a stationary point $ \bar{z} \in S $ of the penalty~\cref{prob:scvx_penalty} is feasible for the original~\cref{prob:scvx_original}, then it is also a KKT point of the original~\cref{prob:scvx_original}, and therefore, \cref{eq:kkt} holds at $\bar{z}$.
\end{theorem}
Although~\cref{thm:exactness} does not suggest a way to find such $\lambda$ and $\tau$, it nevertheless has important theoretical implications. In our current implementation, we select ``sufficiently" large constant $\lambda$'s and $\tau$'s, a strategy that has been shown to work well in practice. Alternatively, the values of $\lambda$ and $\tau$ can be adjusted after each succession, based on the value of the dual solution obtained in the previously solved subproblem. \\
However, it is the second (i.e. ``converse'') part of~\cref{thm:exactness} that is particularly important to our subsequent convergence analysis. Specifically, it guarantees that as long as we can find a stationary point for~\cref{prob:scvx_penalty} that is also feasible for~\cref{prob:scvx_original}, then that point also satisfies the KKT conditions for~\cref{prob:scvx_original}. \\
Now, two important convergence results will be stated. The first one deals with the case of finite convergence, while the second handles situation where an infinite sequence $\{z^k\}$ is generated by the \scvx algorithm.
\begin{theorem} \label{thm:finite_converge}
	Given~\cref{asup:LICQ}, the predicted cost reductions $\Delta L^k$ defined in \cref{eq:DL} are nonnegative for all $k$. Furthermore, $\Delta L^k = 0$ implies that $ z^{k} \in S $, and therefore that $z^{k}$ is a stationary point of the non-convex penalty~\cref{prob:scvx_penalty}.
\end{theorem}
\begin{proof}
	Denote $d$ as the solution to the convex subproblem (i.e.~\cref{prob:scvx_sub}), and note that $d=0$ is always a feasible solution to said problem. Hence we have
	\begin{equation*}
	L^k(d) \leq L^k(0) = J(z^k).
	\end{equation*}
	Therefore, $ \Delta L^k=J(z^k)-L^k(d)\geq 0 $, and $ \Delta L^k=0 $ holds if and only if $ d=0 $ is the minimizer of~\cref{prob:scvx_sub}. When $ \Delta L^k=0 $, one can directly apply~\cref{lem:stationary} and~\cite[Lemma~2]{SCvx_cdc16} to get $ z^k \in S $ (see \cite{SCvx_cdc16} for a more detailed proof).
\end{proof}
The case when $ \{z^k \} $ is an infinite sequence represents a greater challenge. The analysis of this limit process is made more difficult due to the non-differentiability of the penalty function $ J(z) $. To facilitate further analysis, we first note that for any $ z $,
\begin{equation}
J(z+d) = L^k(d) + \orderof(\|d\|), \label{eq:taylor}
\end{equation}
where $ \orderof(\|d\|) $ denotes higher order terms of $ \|d\| $, i.e.,
$$ \lim_{\|d\|\rightarrow 0}{\frac{\orderof(\|d\|)}{\|d\|}} = 0, $$
and is independent of $ z $. This can be verified by writing out the Taylor expansion of $ g_i(z) $ and $ h_j(z) $ in $ J(z) $, and then using the fact that $ \orderof(\|d\|) $ can be taken out of $ |\cdot| $ and $ \max(0,\,\cdot\,) $.

The next lemma is a key preliminary result, and its proof also provides some geometric insights into the \scvx algorithm. The proof is similar to and based on that of~\cite[Lemma 3]{SCvx_cdc16}.
\begin{lemma} \label{lem:main}
	Let $ \bar{z} \in \real^\nz $ be a feasible point, but not a stationary point, of~\cref{prob:scvx_penalty}. Use $N(\bar{z},\bar{\epsilon})$ to denote an open neighborhood of $\bar{z}$ with radius $\bar{\epsilon}$, and let $\mathcal{P}(z,r)$ denote~\cref{prob:scvx_sub} with a linearization evaluated at $z$ and a trust region of radius $r$. Then, for any $c\in(0,1)$, there exist $ \bar{r}>0 $ and $ \bar{\epsilon}>0 $ such that for all $ z \in N(\bar{z},\bar{\epsilon}) $ and $r \in (0,\bar{r}\,]$, any optimal solution $ \dsol $ for $\mathcal{P}(z,r)$ satisfies
	\begin{equation}
	\rho(z,r) \definedas\frac{J(z)-J(z+\dsol)}{J(z)-L^k(\dsol)} \geq c.
	\end{equation}
\end{lemma}
\begin{proof}
	Since $\bar{z}$ is feasible but not stationary, we know that $ 0 \notin \partial J(\bar{z}) $ according to~\cref{lem:stationary}.
	
	From \cref{eq:gen_diff}, it follows that the generalized differential $ \partial J(\bar{z}) $ is the intersection of half spaces, and hence is a closed convex set, which does not include the vector $0$. Hence, using~\cite[Corollary 11.4.2]{rockafellar1970convex} (a strong form of the separation theorem of two closed convex sets, i.e. $0$ and $\partial J(\bar{z})$), there exists a unit vector $ s\in\real^\nz $ and a scalar $ \kappa > 0 $ such that for all $ \nu\in \partial J(\bar{z}) $,
	\begin{equation}
	\nu^T s  \leq -\kappa <0. \label{eq:separation}
	\end{equation}
	Therefore, it follows that
	$$ \max \{\nu^T s \,|\,\nu\in \partial J(\bar{z})\} \leq -\kappa. $$
	The left hand side is exactly the expression for the GDD, as defined in \cref{eq:gdd2}. Therefore, we have
	\begin{equation*}
	dJ(\bar{z},s) \definedas \limsup\limits_{\mathclap{\substack{z\rightarrow \bar{z} \\ r\rightarrow 0^+}}} \frac{J(z+r s)-J(z)}{r} \leq -\kappa.
	\end{equation*}
	This implies that there exist positive $ \bar{r} $ and $ \bar{\epsilon} $ such that for all $ z \in N(\bar{z},\bar{\epsilon}) $ and $r \in (0,\,\bar{r}\,]$,
	\begin{equation} \label{eq:limit_delta}
	\frac{J(z+r s)-J(z)}{r} < -\frac{\kappa}{2}.
	\end{equation}
	%	Another subset of $ D(\bar{z}) $ is $ C(\bar{z}) $ defined in \cref{eq:conic}, so \cref{eq:separation} also holds for all $ \nu \in C(\bar{z}) $. This necessarily implies
	%	\begin{equation} \label{eq:feasible_s}
	%	\langle\nabla h_i(\bar{z}), s\rangle \leq 0, \quad i \in I(\bar{z}),
	%	\end{equation}
	%	which essentially means $ s $ is a feasible direction of the active constraints at $ \bar{z} $.
	Now, assume~\cref{prob:scvx_sub} is solved with $ z\in N(\bar{z},\bar{\epsilon}) $ and $ r\in(0,\,\bar{r}\,] $, producing an optimal solution $ \dsol $. By using \cref{eq:taylor}, the (nonlinear) change realized in $ J $ is
	\begin{align}
	\Delta J(z,\dsol) &=J(z)-J(z+\dsol) \nonumber \\
	&=J(z)-L^k(\dsol)-\orderof(\|\dsol\|) \nonumber \\
	&=\Delta L^k(\dsol)-\orderof(r). \label{eq:delta_j}
	\end{align}
	Thus, the ratio $\rho(z,r)$ is given by
	\begin{equation*}
	\rho(z,r)=\frac{\Delta J(z,\dsol)}{\Delta L^k(\dsol)} = 1-\frac{\orderof(r)}{\Delta L^k(\dsol)}.
	\end{equation*}
	Next, let $ d' = r s \in \real^\nz $. From the definition of $s$, it follows that $\|d'\| = r$, and therefore that $d'$ is within the thrust region of radius $\bar{r}$. Since $ \dsol $ is the optimal solution, we have $ L^k(\dsol)\leq L^k(d') $, which in turn means
	\begin{equation}
	\Delta L^k(\dsol) \geq \Delta L^k(d') . \label{eq:delta_l}
	\end{equation}
	Now, as  $ r \rightarrow 0 $, we will have  $ r \in (0,\,\bar{r}\,]$, and from \cref{eq:limit_delta} we have
	\begin{equation}
	\Delta J(z,d')=J(z)-J(z+d') > \left( \frac{\kappa}{2} \right) r. \label{eq:delta_j_prime}
	\end{equation}
	Replacing $ \dsol $ with $ d' $ in \cref{eq:delta_j} and substituting into \cref{eq:delta_j_prime},
	\begin{equation*}
	\Delta J(z,d')=\Delta L^k(d')-\orderof(r) > \left( \frac{\kappa}{2} \right) r.
	\end{equation*}
	Combining this with \cref{eq:delta_l}, we obtain
	\begin{equation}
	\Delta L^k(\dsol) - \orderof(r) \geq \Delta L^k(d')-\orderof(r) > \left( \frac{\kappa}{2} \right) r. \label{eq:delta_j_star}
	\end{equation}
	Thus $ \Delta L^k(\dsol) > ( \kappa/2 ) r + \orderof(r) > 0$, and the ratio
	\begin{equation}
	\rho(z,r) = 1-\frac{\orderof(r)}{\Delta L^k(\dsol)} > 1-\frac{\orderof(r)}{( \kappa/2 ) r + \orderof(r)}. \label{eq:rho_goes_to_1}
	\end{equation}
	Therefore, as $ r \rightarrow 0 $, $ \rho(z,r) \rightarrow 1 $, and thus for any $c\in(0,1)$, there exists $\bar{r} > 0 $  such that for all $ r\in(0,\,\bar{r}\,]$, $ \rho(z,r) \geq c $ holds.
\end{proof}

\begin{remark}
	An undesirable situation is one where steps are rejected indefinitely (i.e. by producing solutions that result in $\rho(z,r) < \rho_0$). \cref{lem:main} provides an assurance that the \scvx algorithm will not produce such behavior. By contracting $r^k$ sufficiently,~\cref{lem:main} guarantees that the ratio $\rho^k$ will eventually exceed $\rho_0$, and the algorithm will stop rejecting steps.
	%	A frustrating situation can happen if the algorithm keeps rejecting our trial steps indefinitely. Fortunately,~\cref{lem:main} provides an assurance that the \scvx algorithm won't reject the trial steps forever. At some point, usually after reducing $ r^{k} $ several times, the ratio metric $ \rho^k $ will be greater than any prescribed $ c $, hence $ \rho^k \geq \rho_{0} $ will occur and the algorithm will proceed to the next step.
\end{remark}

We are now ready to present our main result. The proof of this theorem is based on that of~\cite[Theorem 4]{SCvx_cdc16}.
\begin{theorem} \label{thm:inf_converge}
	Given~\cref{asup:LICQ}, if the \scvx algorithm generates an infinite sequence $ \{ z^k \} $, then $ \{ z^k \} $ is guaranteed to have limit points. Furthermore, any such limit point, $ \bar{z} $, is a stationary point of the non-convex penalty~\cref{prob:scvx_penalty}.
\end{theorem}
\begin{proof}
	Since we have assumed the feasible region to be convex and compact, by the Bolzano-Weierstrass theorem (see e.g. \cite{rudin1964principles}), there is at least one convergent subsequence $ \{ z^{k_{i}} \} \rightarrow \bar{z} $, which is a guaranteed limit point.
	
	The proof of stationarity of $\bar{z}$ is by contradiction, i.e. we assume that $\bar{z}$ is \underline{not} a stationary point. From~\cref{lem:main},  there exist positive $ \bar{r} $ and $ \bar{\epsilon} $ such that
	$$ \rho(z,r) \geq \rho_0 \quad \forall \; z \in N(\bar{z},\bar{\epsilon}\,) \;\;\text{and}\;\; r\in(0,\,\bar{r}\,], $$
	since $ \rho_0 $ can be chosen arbitrarily small. Without loss of generality, we can suppose the whole subsequence $ \{ z^{k_{i}} \} $ is in $ N(\bar{z},\bar{\epsilon}) $, so that
	\begin{equation} \label{eq:accept}
	\rho(z^{k_{i}},r) \geq \rho_0 \quad \forall\; r\in(0,\,\bar{r}\,].
	\end{equation}
	If the initial trust region radius is less than $ \bar{r} $, then \cref{eq:accept} will be trivially satisfied. On the other hand, if the initial radius is greater than $ \bar{r} $, then the trust region radius may need to be reduced several times by the rejection step in line~\ref{line:reject} of~\cref{algo:SCvx} before condition \cref{eq:accept} is satisfied. For each $ k_i $, let $\hat{r}^{k_i}$ denote the last radius that needs to be reduced, and note that $ \hat{r}^{k_i} > \bar{r} $. Additionally, let $ r^{k_{i}} $ be the trust region radius selected after the last rejection step. Thus, we have~\cref{prob:scvx_sub}, which was obtained by linearizing about $ z^{k_i} $, and which is subject to the trust region constraint $ \|z - z^{k_i}\|\leq r^{k_{i}} $. Then
	\begin{equation*}
	r^{k_{i}} = \hat{r}^{k_i}/\alpha >  \bar{r}/\alpha.
	\end{equation*}
	Since there is also a lower bound $ r_{l} $ on $ r $, we have
	\begin{equation} \label{min_delta}
	%    r^{k_{i}} \geq \max \{r_{l}, \bar{r}/\alpha\} \triangleq \delta.
	r^{k_{i}} \geq \min \{r_{l}, \bar{r}/\alpha\} \triangleq \delta > 0.
	\end{equation}
	Note that condition \cref{eq:accept} can be expressed as
	\begin{equation} \label{eq:r_alter}
	J(z^{k_{i}})-J(z^{k_{i}+1}) \geq \rho_{0} \Delta L^{k_{i}}.
	\end{equation}
	By~\cref{thm:finite_converge}, we have $ \Delta L^{k_{i}} \geq 0 $. Thus, \cref{eq:r_alter} implies that the penalized cost is monotonically decreasing.
	
	Our next goal is to find a lower bound for $ \Delta L^{k_{i}} $. Let $ \dsol $ be the  solution of~\cref{prob:scvx_sub}, linearized about $ \bar{z} $, and with trust region radius $ \delta/2 $ (i.e. $ \|\dsol\| \leq \delta/2 $). Let $ \hat{z}=\bar{z}+\dsol $. Then,
	\begin{equation} \label{eq:norm1}
	\|\hat{z}-\bar{z}\| \leq \delta/2.
	\end{equation}
	Since $ \bar{z} $ is not a stationary point,~\cref{thm:finite_converge} implies
	\begin{equation*}
	\Delta L^k(\dsol)=J(\bar{z}) - L^k(\dsol) \triangleq \theta > 0.
	\end{equation*}
	Consequently, by continuity of $ J $ and $ L^k $, there exists an $ i_0 > 0 $ such that for all $ i \geq i_0 $
	\begin{align}
	J(z^{k_{i}}) - L^{k_i}(\hat{z}-z^{k_{i}}) &> \theta/2, \text{ and}  \label{eq:r_continuity} \\
	\|z^{k_{i}} - \bar{z}\| &< \delta/2.\label{eq:norm2}
	\end{align}
	From \cref{eq:norm1} and \cref{eq:norm2}, for all $ i \geq i_0 $, we have
	\begin{equation} \label{eq:tr_ki}
	\|\hat{z}-z^{k_{i}}\| \leq \|\hat{z}-\bar{z}\| + \|z^{k_{i}} - \bar{z}\| < \delta \leq r^{k_{i}},
	\end{equation}
	where the last inequality comes from \cref{min_delta}.
	
	Defining $ \hat{d}^{k_{i}}=\hat{z}-z^{k_{i}} $, \cref{eq:tr_ki} implies that $ \hat{d}^{k_{i}} $ is a feasible solution for the convex subproblem (i.e.~\cref{prob:scvx_sub}) at $ (z^{k_{i}},r^{k_{i}}) $ when $ i \geq i_0 $. Then if $ d^{k_{i}} $ is the optimal solution to this subproblem, we have $ L^{k_i}(d^{k_{i}}) \leq L^{k_i}(\hat{d}^{k_{i}}) $, so that
	%	\begin{align}
	%		\Delta L^k(d^{k_{i}}) &= J(x^{k_{i}}) - L^k(d^{k_{i}}) \nonumber \\
	%		&\geq J(x^{k_{i}}) - L^k(\hat{d}^{k_{i}}) \nonumber \\
	%		&> \theta/2. \label{eq:theta/2}
	%	\end{align}
	\begin{align} \label{eq:theta/2}
	\Delta L^{k_i}(d^{k_{i}}) &= J(z^{k_{i}}) - L^{k_i}(d^{k_{i}}) \notag \\
	&\geq J(z^{k_{i}}) - L^{k_i}(\hat{d}^{k_{i}}) \notag \\
	&> \theta/2.
	\end{align}
	The last inequality is due to \cref{eq:r_continuity}. Combining \cref{eq:r_alter} and \cref{eq:theta/2}, we obtain the following for all $ i \geq i_0 $:
	\begin{equation} \label{eq:last_theta/2}
	J(z^{k_{i}})-J(z^{k_{i}+1}) \geq \rho_{0} \theta/2.
	\end{equation}
	However, since $k_i+1 \leq k_{(i+1)}$, and the penalized cost is not increasing due to \cref{eq:r_alter}, we have $ J(z^{k_{i}+1}) \geq J(z^{k_{(i+1)}}) $, and thus
	\begin{align*}
	\sum_{i=1}^{\infty} \left( J(z^{k_{i}})-J(z^{k_{i}+1}) \right) &\leq 
	\sum_{i=1}^{\infty} \left( J(z^{k_{i}})-J(z^{k_{(i+1)}}) \right) \\
	&=J(z^{k_{1}})-J(\bar{z}) \leq \infty.
	\end{align*}
	Therefore, the series is convergent, and necessarily
	$$ J(z^{k_{i}})-J(z^{k_{i}+1}) \;\rightarrow\; 0, $$
	which contradicts \cref{eq:last_theta/2}. This contradiction implies that every limit point $ \bar{z} $ is a stationary point of~\cref{prob:scvx_penalty}.
\end{proof}
Now, combining \cref{thm:finite_converge} and \cref{thm:inf_converge} and using the ``converse" part of \cref{thm:exactness}, we can summarize the final result as follows.
\begin{theorem}[\textbf{Global Weak Convergence}] \label{thm:global_converge}
	Given~\cref{asup:LICQ}, regardless of initial conditions, the \scvx algorithm (\cref{algo:SCvx}) always has limit points, and any limit point, $ \bar{z} $, is a stationary point of the non-convex penalty~\cref{prob:scvx_penalty}. Furthermore, if $ \bar{z} $ is feasible for the original non-cnvex~\cref{prob:scvx_original}, then it is a KKT point of~\cref{prob:scvx_original}.
\end{theorem}
The weak convergence result is inline with classical nonlinear optimization analysis, but does not prevent oscillation among several limit points. In the next section, we will see that this result can in fact be strengthened to single point convergence under some mild additional assumptions. Further, on its own \cref{thm:global_converge} does not guarantee feasibility of the converged solution, even if the original non-convex problem, \cref{prob:scvx_original}, may have feasible regions. However, it is important to point out that this scenario is rarely observed in practice, and when it actually occurs, trying different initial guesses would often resolve it.

\subsection{Strong Convergence}
In this section, we will show that the \scvx algorithm indeed converges to a single limit point under some mild additional assumptions. For the purpose of this section, it is not necessary to distinguish non-convex inequality constraints $ g_{i}(z) \leq 0, \forall \, i\in\ncvxineq $ and convex inequality constraints $ h_{j}(z) \leq 0, \forall \, j\in\cvxineq $. Thus we will use $ g_{i}(z) \leq 0, \forall \, i\in \{\Nncvxeq+1, \ldots, \Nncvxineq+\Ncvxineq\} $ to denote all inequality constraints. For strong convergence, first we need a slightly stronger smoothness assumption.
\begin{assumption} \label{asup:lipschitz_grad}
	$ g_i(z), \forall \, i = 0, \ldots, \Nncvxineq+\Ncvxineq $ have Lipschitz continuous gradients, that is, $ \exists \, L_i \geq 0 $, s.t.
	\begin{equation*}
	\| \nabla g_i(z_2) - \nabla g_i(z_1) \| \leq L_i \| z_2 - z_1 \|.
	\end{equation*}
\end{assumption}
%Note that this assumption is just slightly stronger than the standard $ C^1 $ assumption.

Using the simplified notation, we can rewrite the original non-convex (penalty) cost functions as
\begin{equation}
J(z) = g_0(z) + \sum_{i=1}^{e} \lambda_i |g_i(z)| + \sum_{i=\Nncvxeq+1}^{\Nncvxineq+\Ncvxineq} \lambda_i \max \big(0,g_i(z)\big),  \label{eq:simple_pen_cost}
\end{equation}
and the linearized (penalty) cost function at $ k $th iteration as
\begin{equation}
\begin{aligned}
L^k(d) = &\, g_0(z^k) + \nabla g_0(z^k)^T d + \sum_{i=1}^{e} \lambda_i |g_i(z^k) + \nabla g_i(z^k)^T d| \\
&+ \sum_{i=\Nncvxeq+1}^{\Nncvxineq+\Ncvxineq} \lambda_i \max \big(0,g_i(z^k) + \nabla g_i(z^k)^T d\big).  \label{eq:simple_lin_cost}
\end{aligned}
\end{equation}

The key assumption used in \cite{attouch2013convergence} to establish single point convergence is that the function been optimized satisfies the (nonsmooth) Kurdyka--{\L}ojasiewicz (KL) property~\cite[Definition~2.4]{attouch2013convergence}, which means, roughly speaking, that the functions under consideration are ``sharp up to a reparametrization". Real semi-algebraic functions provide a very rich class of functions satisfying the KL property. Many other functions may also satisfy this property, among which an important class is given by functions definable in an o-minimal structure~\cite{kurdyka1998gradients}. One can verify that many real-world optimal control problems indeed have KL property, especially by attesting its semi-algebracity. Therefore, it is reasonable to have the following assumption:
\begin{assumption} \label{asup:KL}
	The penalized cost function $ J(z) $ in \cref{eq:simple_pen_cost} have KL property~\cite[Definition~2.4]{attouch2013convergence}.
\end{assumption}
Given \cref{asup:KL}, \cite[Section~2.3]{attouch2013convergence} establishes three conditions (referred as H1--H3) one can check to ensure strong convergence of an optimization scheme. Note that H3, the \emph{continuity condition}, has already been proved in \cref{thm:inf_converge}. Now, let us proceed to verify the other two conditions in the context of \scvx algorithm.
\begin{condition}[\textbf{Sufficient Decrease}] \label{cond:H1}
	At iteration $ k $, let $ z^k $ be the current point, and $ z^{k+1} $ be the accepted solution by the \scvx algorithm, then the penalized cost function $ J(\cdot) $ in \cref{eq:simple_pen_cost} satisfies
	\begin{equation}
	J(z^k) - J(z^{k+1}) \geq a \|z^{k+1} - z^k\|^2, \label{eq:suff_dec}
	\end{equation}
	where $ a > 0 $ is some constant.
\end{condition}
\begin{condition}[\textbf{Relative Error}] \label{cond:H2}
	At iteration $ k $, let $ z^k $ be the current point, and $ z^{k+1} $ be the accepted solution by the \scvx algorithm, then given~\cref{asup:lipschitz_grad}, $ \exists \, \omega^{k+1} \in \partial J(z^{k+1}) $, s.t.
	\begin{equation}
	\| \omega^{k+1} \| \leq b \|z^{k+1} - z^k \|,  \label{eq:relative_error}
	\end{equation}
	where $ b > 0 $ is some constant, and $ \partial J(\cdot) $ is the generalized differential defined in~\cref{eq:gen_diff}.
\end{condition}
The intuition behind these two conditions can be found in \cite[p.~92]{attouch2013convergence}, and roughly speaking, \cref{cond:H1} measures the quality of a descent step, while \cref{cond:H2} reflects relative inexact optimality conditions for subproblems. The proof of \cref{cond:H1} is relatively straightforward, and is given as follows.
\begin{proof}[Proof of~\cref{cond:H1}]
	Note that in \cref{eq:delta_j}, $ d^* = z^{k+1} - z^k $, then from \cref{eq:delta_j}, \cref{eq:delta_j_star} and the trust region constraint $ \|d^*\| \leq r $, we have that
	\begin{equation}
	J(z^k) - J(z^{k+1}) \geq \frac{\kappa}{2}r \geq \frac{\kappa}{2}\|d^*\|. \label{eq:delta_j_k}
	\end{equation}
	Since $ z \in Z = X \times U $, and $ Z $ is compact, $ \|d^*\| = \|z^{k+1} - z^k\| $ is bounded. Let $ D = \underset{z_1, z_2 \in Z}{\max} \|z_1 - z_2\| $, then $ \|d^*\| \leq D, \forall \, d $.
	
	Now, let $ a = \frac{\kappa}{2D} $, then from \ref{eq:delta_j_k}, we have
	\begin{equation*}
	J(z^k) - J(z^{k+1}) \geq \frac{\kappa}{2} \frac{1}{\|d^*\|} \|d^*\|^2 \geq \frac{\kappa}{2D} \|d^*\|^2 = a \|d^*\|^2 = a \|z^{k+1} - z^k\|^2,
	\end{equation*}
	which is exactly \cref{eq:suff_dec}.
\end{proof}

The verification of \cref{cond:H2} requires a bit more work and a constructive approach. We will leverage convexity of the generalized subdifferential and some special structures of the two types of non-smoothness, $ | \cdot | $ and $ \max(0, \cdot) $ presented in \cref{eq:simple_pen_cost,eq:simple_lin_cost}. Namely, $ \partial |g_i(z)| $ is symmetric with respect to $ 0 $, and in fact a convex hull spanned by $ -\nabla g_i(z) $ and $ \nabla g_i(z) $. Thus we have $ \forall \, \omega_i \in \partial |g_i(z)|, i = 1, \ldots, e $, $ \exists \, \alpha_i \in [0,1] $, s.t.
\begin{equation}
\begin{aligned}
\omega_i &= \alpha_i \nabla g_i(z) + (1-\alpha_i) (-\nabla g_i(z)), \\
&= (2\alpha_i -1)\nabla g_i(z)   \label{partial_equality}.
\end{aligned}
\end{equation}
Similarly, $ \partial \max \big(0,g_i(z)\big) $ is a convex hull of $ 0 $ and $ \nabla g_i(z) $, and we also have $ \forall \, \omega_i \in \partial \max \big(0,g_i(z)\big), i = \Nncvxeq+1, \ldots, \Nncvxineq+\Ncvxineq $, $ \exists \, \alpha_i \in [0,1] $, s.t.
\begin{equation}
\omega_i = \alpha_i \nabla g_i(z) + (1-\alpha_i) 0 = \alpha_i \nabla g_i(z)   \label{partial_inequality}.
\end{equation}

Next we present two technical lemmas asserting the close proximity of the generalized differentials at $ z^{k+1} $ and their linear approximations around $ z^k $.
\begin{lemma}  \label{lem:diff_eq}
	For any $ \nu_i^k \in \partial_d |g_i(z^k) + \nabla g_i(z^k)^T d| $, $ i = 1, \ldots, e $, there exists $ \omega_i^{k+1} \in \partial |g_i(z^{k+1})| $ and constant $ c_i \geq 0 $ s.t.
	\begin{equation*}
	\| \omega_i^{k+1} - \nu_i^k \| \leq c_i \| z^{k+1} - z^k \|.
	\end{equation*}
\end{lemma}
\begin{proof}
	The distance between the two generalized differentials is captured by the distance between the two vectors, $ \omega_i^{k+1} $ and $ \nu_i^k $, and we have
	\begin{equation}
	\label{partial_diff}
	\begin{aligned}
	\| \omega_i^{k+1} - \nu_i^k \| &= \| \omega_i^{k+1} - \nabla g_i(z^{k+1}) + \nabla g_i(z^{k+1}) - \nu_i^k + \nabla g_i(z^k) - \nabla g_i(z^k) \|  \\
	&\leq \| \omega_i^{k+1} - \nabla g_i(z^{k+1}) - \nu_i^k + \nabla g_i(z^k) \| + \| \nabla g_i(z^{k+1}) - \nabla g_i(z^k)\|.
	\end{aligned}
	\end{equation}
	Now, using \cref{partial_equality}, we have $ \omega_i^{k+1} = (2\alpha_i -1)\nabla g_i(z^{k+1}) $, and $ \nu_i^k = (2\bar{\alpha}_i -1)\nabla g_i(z^k) $. Note that $ \alpha_i \in [0,1] $ is free, while $ \bar{\alpha}_i \in [0,1] $ is fixed for a given $ \nu_i^k $. Therefore, we have
	\begin{equation*}
	\| \omega_i^{k+1} - \nu_i^k \| \leq \| 2(\alpha_i -1)\nabla g_i(z^{k+1}) - 2(\bar{\alpha}_i -1)\nabla g_i(z^k) \| + \| \nabla g_i(z^{k+1}) - \nabla g_i(z^k)\|.
	\end{equation*}
	As illustrated in \cref{fig:convex}, we may choose $ \alpha_i = \bar{\alpha}_i $, and get
	\begin{align*}
	\| \omega_i^{k+1} - \nu_i^k \| &\leq 2(1 - \bar{\alpha}_i) \| \nabla g_i(z^{k+1}) - \nabla g_i(z^k)\| + \| \nabla g_i(z^{k+1}) - \nabla g_i(z^k)\|  \\
	&= (3 - 2\bar{\alpha}_i) \| \nabla g_i(z^{k+1}) - \nabla g_i(z^k)\|.
	\end{align*}
	By~\cref{asup:lipschitz_grad}, we have
	\begin{displaymath}
	\| \omega_i^{k+1} - \nu_i^k \| \leq (3 - 2\bar{\alpha}_i) L_i \| z^{k+1} - z^k \| := c_i  \| z^{k+1} - z^k \|.
	\end{displaymath}
\end{proof}

\begin{figure}[!ht]
	\begin{centering}
		\captionsetup{width=0.6\textwidth}
		\includegraphics[width=0.6\textwidth]{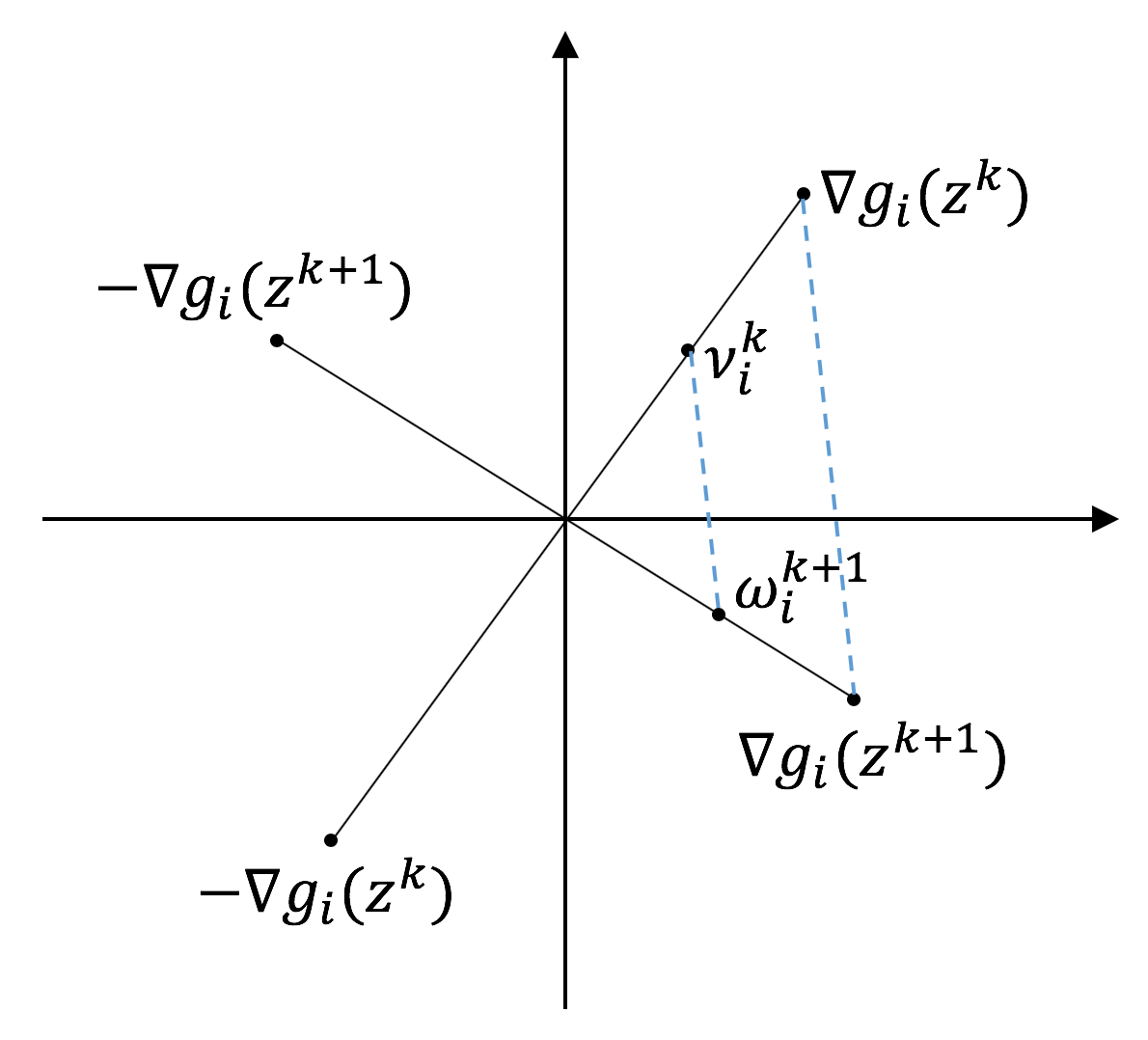}
		\caption{\emph{In the proof of \cref{lem:diff_eq}, by choosing $ \alpha_i = \bar{\alpha}_i $, we are effectively setting $ \omega_i^{k+1} $ to have the same convex combination factor as $ \nu_i^k $.}}
		\label{fig:convex}
	\end{centering}
\end{figure}

\begin{lemma}  \label{lem:diff_ineq}
	For any $ \nu_i^k \in \partial_d \max \big(0,g_i(z^k) + \nabla g_i(z^k)^T d\big) $, $ i = \Nncvxeq+1, \ldots, \Nncvxineq+\Ncvxineq $, there exists $ \omega_i^{k+1} \in \partial \max \big(0,g_i(z^{k+1})\big) $ and constant $ c_i \geq 0 $ s.t.
	\begin{equation*}
	\| \omega_i^{k+1} - \nu_i^k \| \leq c_i \| z^{k+1} - z^k \|.
	\end{equation*}
\end{lemma}
\begin{proof}
	We still have \cref{partial_diff} in this case. 
	Now, using \cref{partial_inequality}, we have $ \omega_i^{k+1} = \alpha_i \nabla g_i(z^{k+1}) $, and $ \nu_i^k = \bar{\alpha}_i \nabla g_i(z^k) $. Again, note that $ \alpha_i \in [0,1] $ is free, while $ \bar{\alpha}_i \in [0,1] $ is fixed for a given $ \nu_i^k $. Therefore, from \cref{partial_diff} we have
	\begin{equation*}
	\| \omega_i^{k+1} - \nu_i^k \| \leq \| (\alpha_i -1)\nabla g_i(z^{k+1}) - (\bar{\alpha}_i -1)\nabla g_i(z^k) \| + \| \nabla g_i(z^{k+1}) - \nabla g_i(z^k)\|.
	\end{equation*}
	Again, choose $ \alpha_i = \bar{\alpha}_i $, and we get
	\begin{align*}
	\| \omega_i^{k+1} - \nu_i^k \| &\leq (1 - \bar{\alpha}_i) \| \nabla g_i(z^{k+1}) - \nabla g_i(z^k)\| + \| \nabla g_i(z^{k+1}) - \nabla g_i(z^k)\|  \\
	&= (2 - \bar{\alpha}_i) \| \nabla g_i(z^{k+1}) - \nabla g_i(z^k)\|.
	\end{align*}
	By~\cref{asup:lipschitz_grad}, we have
	\begin{displaymath}
	\| \omega_i^{k+1} - \nu_i^k \| \leq (2 - \bar{\alpha}_i) L_i \| z^{k+1} - z^k \| := c_i  \| z^{k+1} - z^k \|.
	\end{displaymath}
\end{proof}

Now we are ready to present the proof of~\cref{cond:H2}, which makes use of the optimality conditions of the convex subproblem.
\begin{proof}[Proof of~\cref{cond:H2}]
	First note that
	\begin{equation}
	\omega^{k+1} = \nabla g_0(z^{k+1}) + \sum_{i=1}^{e} \lambda_i \omega_i^{k+1} + \sum_{i=\Nncvxeq+1}^{\Nncvxineq+\Ncvxineq} \lambda_i \omega_i^{k+1},  \label{omega}
	\end{equation}
	where $ \omega_i^{k+1} \in \partial |g_i(z^{k+1})| $, $ i = 1, \ldots, e$ and $ \omega_i^{k+1} \in \partial \max \big(0,g_i(z^{k+1})\big) $, $ i = \Nncvxeq+1, \ldots, \Nncvxineq+\Ncvxineq $.
	Now let us turn to the stationarity condition for the convex subproblem:
	\begin{equation}
	0 \in \partial_d L^k(d) + \mathcal{N}_{\|d\| \leq r^k}(d)  \label{sub_kkt}
	\end{equation}
	where $ \mathcal{N}_{\|d\| \leq r^k}(d) $ is the normal cone of set $ \|d\| \leq r^k $ at $ d $, which turns out to be $ c\,d $, i.e. $ c\,(z^{k+1} - z^k) $, where $ c\geq 0 $ is a constant. From the formula of $ L^k(d) $ in \cref{eq:simple_lin_cost}, condition in \cref{sub_kkt} is equivalent to
	\begin{equation}
	\exists \, \nu^k \in \partial_d L^k(d), \quad \textnormal{s.t. } 0 = \nu^k + c\,(z^{k+1} - z^k), \label{sub_kkt_eq}
	\end{equation}
	where $ \nu^k = \nabla g_0(z^k) + \sum_{i=1}^{e} \lambda_i \nu_i^k + \sum_{i=\Nncvxeq+1}^{\Nncvxineq+\Ncvxineq} \lambda_i \nu_i^k $, where $ \nu_i^k \in \partial_d |g_i(z^k) + \nabla g_i(z^k)^T d|$, $ i = 1, \ldots, e$ and $ \nu_i^k \in \partial_d \max \big(0,g_i(z^k) + \nabla g_i(z^k)^T d\big) $, $ i = \Nncvxeq+1, \ldots, \Nncvxineq+\Ncvxineq $. Therefore, \cref{sub_kkt_eq} can be written as
	\begin{align*}
	0 = &\, \nabla g_0(z^k) + \sum_{i=1}^{e} \lambda_i \nu_i^k + \sum_{i=\Nncvxeq+1}^{\Nncvxineq+\Ncvxineq} \lambda_i \nu_i^k + c\,(z^{k+1} - z^k)  \\
	\Leftrightarrow \quad 0 = &\, \nabla g_0(z^k) - \nabla g_0(z^{k+1}) + \nabla g_0(z^{k+1})  \\
	&+ \sum_{i=1}^{e} \lambda_i (\nu_i^k - \omega_i^{k+1} + \omega_i^{k+1}) \\
	&+ \sum_{i=\Nncvxeq+1}^{\Nncvxineq+\Ncvxineq} \lambda_i (\nu_i^k - \omega_i^{k+1} + \omega_i^{k+1}) + c\,(z^{k+1} - z^k)
	\end{align*}
	\begin{align*}
	\Leftrightarrow \quad &\nabla g_0(z^{k+1}) + \sum_{i=1}^{e} \lambda_i \omega_i^{k+1} + \sum_{i=\Nncvxeq+1}^{\Nncvxineq+\Ncvxineq} \lambda_i \omega_i^{k+1} \\
	& \quad =\, \nabla g_0(z^k) - \nabla g_0(z^{k+1}) + \sum_{i=1}^{e} \lambda_i (\nu_i^k - \omega_i^{k+1})  \\
	& \qquad + \sum_{i=\Nncvxeq+1}^{\Nncvxineq+\Ncvxineq} \lambda_i (\nu_i^k - \omega_i^{k+1}) + c\,(z^{k+1} - z^k)
	\end{align*}
	The left-hand-side of the above equation is exactly $ \omega^{k+1} $ as in \cref{omega}, and apply~\cref{asup:lipschitz_grad} and~\cref{lem:diff_eq} and~\cref{lem:diff_ineq} to the right-hand-side, we get
	\begin{align*}
	\| \omega^{k+1} \| \leq & \, \| \nabla g_0(z^k) - \nabla g_0(z^{k+1}) \|  \\
	&+ \sum_{i=1}^{\Nncvxineq+\Ncvxineq} \lambda_i \| \nu_i^k - \omega_i^{k+1} \| + c\, \| z^{k+1} - z^k \|  \\
	\leq & \, (L_0 + \sum_{i=1}^{\Nncvxineq+\Ncvxineq} \lambda_i c_i + c) \| z^{k+1} - z^k \|.
	\end{align*}
	where $ c_i = (3 - 2\bar{\alpha}_i) L_i $ for $ i = 1, \ldots, e $ and $ c_i = (2 - \bar{\alpha}_i) L_i $ for $ i = \Nncvxeq+1, \ldots, \Nncvxineq+\Ncvxineq $, where $ \bar{\alpha}_i $ is the convex combination factor for $ \nu_i^k $.
	Let $ b = L_0 + \sum_{i=1}^{\Nncvxineq+\Ncvxineq} \lambda_i c_i + c $, we have \cref{eq:relative_error}.
\end{proof}

Now that we verified the strong convergence conditions, \cref{cond:H1} and \cref{cond:H2}, by using~\cite[Theorem~2.9]{attouch2013convergence} and the weak convergence results in the previous section~\ref{thm:global_converge}, we are in the position to claim the following:
\begin{theorem}[\textbf{Global Strong Convergence}] \label{thm:strong_converge}
	Suppose~\cref{asup:LICQ,asup:lipschitz_grad,asup:KL} hold, then regardless of initial conditions, the sequence $ \{z^k\} $ generated by \scvx algorithm (\cref{algo:SCvx}) always converges to a single limit point, $ \bar{z} $, and $ \bar{z} $ is a stationary point of the non-convex penalty~\cref{prob:scvx_penalty}. Furthermore, if $ \bar{z} $ is feasible for the original~\cref{prob:scvx_original}, then it is a KKT point of~\cref{prob:scvx_original}.
\end{theorem}
	
\section{Superlinear Convergence Rate} \label{sec:superlinear}

In this section, we will show that the \scvx algorithm converges not only globally, but also superlinearly under some additional mild assumptions. Moreover, we will see that the superlinear rate of convergence is enabled by the structure of the underlying optimal control problem. In other words, the \scvx algorithm is specifically tailored to solve non-convex optimal control problems, and thus enjoys a faster convergence rate when compared to generic nonlinear programming methods, which often converge linearly, if at all.

First we note that at each succession of~\cref{algo:SCvx}, we are solving a convex programming problem, which is best solved using IPMs that employ self-dual embedding technique introduced by \cite{goldman1956theory}. A particular advantage of this approach is that it always produces a strictly complementary solution, thus satisfying the following assumption:
\begin{assumption}[\textbf{Strict Complementary Slackness}] \label{asup:SCS}
	In addition to the KKT conditions in~\cref{thm:kkt}, we assume that the following conditions are satisfied at the local optimum $ \bar{z} $,
	\begin{align*}
	\bar{\mu}_i &> 0, \quad \forall \, i \in I_{ac}(\bar{z}), \\
	\bar{\sigma}_j &> 0, \quad \forall \, j \in J_{ac}(\bar{z}).
	\end{align*}
\end{assumption}
The next assumption leverages the structure of optimal control problems, and is crucial in subsequent analysis.
\begin{assumption}[\textbf{Binding}] \label{asup:active}
	Let $ z^k \rightarrow \bar{z} $. There are at least $ \nc(\tf-1) $ binding constraints in $ g_{i}(\bar{z})\leq 0, i\in\ncvxineq$ and $ h_{j}(\bar{z})\leq 0, j\in\cvxineq$. That is,
	$$ |I_{ac}(\bar{z})|+|J_{ac}(\bar{z})| \geq \nc(\tf-1). $$
\end{assumption}
Optimal control problems often observes the \emph{bang-bang principle}, provided that the Hamiltonian is affine in controls, the control set is a convex polyhedron, and there are no \emph{singular arcs} \cite{sussmann1983lie}. Linear systems with these properties are referred as \emph{normal} (see Corollary 7.3 of \cite{berkovitz1974optimal}). For nonlinear systems, the non-singular condition can be checked by sequentially examine the \emph{Lie bracket} of the system dynamics (see Section 4.4.3 of \cite{liberzon2012calculus}). If the optimal control is indeed \emph{bang-bang}, then \cref{asup:active} is obviously satisfied. Once classic example is the \emph{bang-bang} solution obtained by solving a minimum time optimal control problem \cite{lasalle1953study}. The algorithm proposed in \cite{harris2014minimum} can also be used to ensure \emph{bang-bang} property of the solutions.

Note that even if some control constraints are inactive at the optimal solution, as long as there are an equal or greater number of active state constraints, then \cref{asup:active} still holds true. An interesting example of such a case is that of the maximum-divert planetary landing problem \cite{harris2014maximum} containing both control constraints and velocity constraints. In this example, the control constraints are inactive only when the velocity constraints are activated.

Nevertheless, given the special structure of optimal control problems and the properties of their solutions, we have the following lemma:
\begin{lemma} \label{lem:basis}
	Given~\cref{asup:LICQ,asup:active}, the gradient set of active constraints $ G_{ac}(\bar{z}) $ defined in~\cref{eq:G_ac} contains a basis of $ \real^\nz $, where $ \nz = \nx \tf + \nc(\tf-1) $. In other words, there are at least $ \nz $ linearly independent vectors in $ G_{ac}(\bar{z}) $.
\end{lemma}
\begin{proof}
	From~\cref{asup:LICQ,asup:active}, we have at least $ \nc(\tf-1) $ linearly independent gradient vectors from active control constraints, In addition, from the formulation of the original optimal control problem, we know that there are at least $ \nx \tf $ equality constraints due to the system dynamics, and their gradient vectors are also linearly independent by~\cref{asup:LICQ}. Therefore, $ G_{ac}(\bar{z}) $ has at least $ \nz $ linearly independent vectors.
\end{proof}
Next, we prove several technical lemmas that are instrumental in obtaining the final convergence rate result. In this section, we will use $ \{z^k\} \rightarrow \bar{z} $ broadly to denote weak convergence (i.e. $ \{z^k\} $ is the subsequence converged to $ \bar{z} $), but if~\cref{asup:lipschitz_grad,asup:KL} are also satisfied, then the same notation means strong convergence (to a single limit point).
%
%\begin{definition}[\textbf{Monotonicity in $ \mathbb{R}^n $}] \label{def:monotone}
%	We say a sequence $ y^k = [y_1^k, y_2^k, \ldots, y_n^k] $ is \emph{monotonic} in $ \mathbb{R}^n $ if each coordinate $ y_i^k $ is monotonic in $ \mathbb{R} $.
%\end{definition}
%%
%\begin{lemma} \label{lem:monotone}
%	Every  bounded sequence in $ \mathbb{R}^n $ has a convergent subsequence, which is also monotonic in the sense of~\cref{def:monotone}.
%\end{lemma}
%
%\begin{proof}
%	The construction of such subsequence is exactly the same as in the proof of the Bolzano-Weierstrass theorem (see e.g. \cite{rudin1964principles}).
%\end{proof}
%
\begin{lemma} \label{lem:limit}
	Let $ \{y^k\} $ be a sequence in $ \real^\nz $ such that $ \{y^k\} \neq \bar{z} $ and $ \{y^k\}\rightarrow \bar{z} $. Then, there exists $ \xi \in \real^\nz, \|\xi\|=1 $, such that for any function $ g(\cdot) \in C^1: \, \real^\nz\rightarrow \real $, we have a subsequence denoted by $ k_{s}, s = 1, 2, \ldots $, such that
	\begin{equation} \label{eq:lim_sub}
	\lim\limits_{k_{s}\rightarrow \infty} \frac{g(y^{k_{s}})-g(\bar{z})}{\|y^{k_{s}} - \bar{z}\|} = \nabla g(\bar{z})^T \xi.
	\end{equation}
\end{lemma}
\begin{proof}
	Define $ \xi^k \definedas \frac{{y^k - \bar{z}}}{{\|y^k - \bar{z}\|}} $, such that $ \| \xi^k \| = 1 $. Clearly $ \{\xi^k\} $ is bounded, and by Bolzano-Weierstrass theorem (see e.g. \cite{rudin1964principles}), this sequence has a convergent subsequence $ \{\xi^{k_{s}}\} \rightarrow \xi $, i.e., the limit of the left hand side of \cref{eq:lim_sub} exists, and
	\begin{equation}
	\xi = \lim\limits_{k_{s}\rightarrow \infty} \xi^{k_{s}} = \lim\limits_{k_{s}\rightarrow \infty} \frac{y^{k_{s}}-\bar{z}}{\|y^{k_{s}} - \bar{z}\|},
	\label{eq:xi}
	\end{equation}
	and $ \|\xi\| = \|\xi^{k_{s}}\| = 1 $. Let $ y^{k_{s}} = \bar{z} + \omega^{k_{s}} \xi^{k_{s}} $, where $ \omega^{k_{s}} = \|y^{k_{s}} - \bar{z}\|, $ and $ \omega^{k_{s}}\rightarrow 0 $, as $ k_{s} \rightarrow \infty $. Then, we have
	\begin{equation*}
	\lim\limits_{k_{s}\rightarrow \infty} \frac{g(y^{k_{s}})-g(\bar{z})}{\|y^{k_{s}} - \bar{z}\|} = \lim\limits_{k_{s}\rightarrow \infty} \frac{g(\bar{z} + \omega^{k_{s}} \xi^{k_{s}})-g(\bar{z})}{\|\omega^{k_{s}} \xi^{k_{s}}\|}.
	\end{equation*}
	Let $ \theta^{k_{s}} \definedas \xi^{k_{s}} - \xi $, such that $ \theta^{k_{s}}\rightarrow 0 $ as $ k_{s} \rightarrow \infty $. Now we have
	\begin{align}
	\lim\limits_{k_{s}\rightarrow \infty} \frac{g(y^{k_{s}})-g(\bar{z})}{\|y^{k_{s}} - \bar{z}\|} %&= \lim\limits_{k_{s}\rightarrow \infty} \frac{g(\bar{z} + \omega^{k_{s}} \xi^{k_{s}})-g(\bar{z})}{\omega^{k_{s}}} \notag \\
	&= \lim\limits_{k_{s}\rightarrow \infty} \frac{g(\bar{z} + \omega^{k_{s}} \xi + \omega^{k_{s}} \theta^{k_{s}})-g(\bar{z})}{\omega^{k_{s}}} \notag \\
	&= \lim\limits_{k_{s}\rightarrow \infty} \frac{g(\bar{z} + \omega^{k_{s}} \xi + \omega^{k_{s}} \theta^{k_{s}})-g(\bar{z} + \omega^{k_{s}} \xi)}{\omega^{k_{s}}} \notag \\
	&\quad + \lim\limits_{k_{s}\rightarrow \infty} \frac{g(\bar{z} + \omega^{k_{s}} \xi)-g(\bar{z})}{\omega^{k_{s}}}. \label{eq:limit}
	\end{align}
	Since $ g(\cdot) \in C^1 $, we apply the mean value theorem, such that the first term of \cref{eq:limit} becomes
	\begin{equation*}
	\lim\limits_{k_{s}\rightarrow \infty} \frac{\omega^{k_{s}} {\theta^{k_{s}}}^T \nabla g(\bar{z} + \omega^{k_{s}} \xi + \omega^{k_{s}} \bar{\theta}^{k_{s}})}{\omega^{k_{s}}} = \lim\limits_{k_{s}\rightarrow \infty} \nabla g(\bar{z})^T \theta^{k_{s}} = 0,
	\end{equation*}
	where $ \bar{\theta}^{k_{s}} $ denotes a point that lies on the line segment between $ 0 $ and $ \theta^{k_{s}} $. Therefore, \cref{eq:limit} becomes
	\begin{equation*}
	\lim\limits_{k_{s}\rightarrow \infty} \frac{g(y^{k_{s}})-g(\bar{z})}{\|y^{k_{s}} - \bar{z}\|} = \lim\limits_{k_{s}\rightarrow \infty} \frac{g(\bar{z} + \omega^{k_{s}} \xi)-g(\bar{z})}{\omega^{k_{s}}} = \nabla g(\bar{z})^T \xi,
	\end{equation*}
	by definition of the directional derivative.
\end{proof}

\begin{lemma} \label{lem:growth}
	Let $ \{z^k\} $ be the convergent sequence generated by the \scvx algorithm, and $ \{z^k\} \rightarrow \bar{z} $. Assume $ \bar{z} $ is feasible to the original problem,~\cref{prob:scvx_original}, then under~\cref{asup:LICQ,asup:SCS,asup:active}, there exist $ \beta>0 $ and $ \delta>0 $ such that
	$$ \forall z \in N(\bar{z}, \delta) \definedas \left\lbrace z \, | \, \|z - \bar{z}\| \leq \delta\right\rbrace, $$
	we have
	\begin{equation}
	J(z) - J(\bar{z}) \geq \beta \|z - \bar{z}\|, \label{eq:growth}
	\end{equation}
	where $ J(z) $ is the penalty cost defined in~\cref{eq:pen_cost}.
\end{lemma}
\begin{proof}
	We will prove by contradiction. Assume the statement is false. It means that for a given diminishing sequence $ \{\varepsilon^k\}\rightarrow 0,\; k=1,2,\,\ldots\, $, there exists sequence $ \{y^k\} \neq \bar{z} $ and $ \{y^k\}\rightarrow \bar{z},\; k=1,2,\,\ldots\, $, such that
	\begin{equation}
	J(y^k) - J(\bar{z}) \leq \varepsilon^k \|y^k - \bar{z}\|. \label{eq:y^k}
	\end{equation}
	Now let $ \{y^{k_{s}}\} $ denote a subsequence of $ \{y^k\} $, such that \cref{eq:xi} holds and $ \|\xi\|=1 $.
	Since $ \bar{z} $ is assumed to be feasible, we have $ J(\bar{z})=g_0(\bar{z}) $. Therefore, \cref{eq:y^k} with respect to $ k_{s} $ becomes
	\begin{align*}
	&g_0(y^{k_{s}}) + \sum_{i\in\ncvxeq}\lambda_i |g_{i}(y^{k_{s}})| + \sum_{i\in\ncvxineq}\lambda_i \max\big(0,g_{i}(y^{k_{s}})\big) \\
	&+ \sum_{j\in\cvxineq}\tau_j \max\big(0,h_{j}(y^{k_{s}})\big) - g_0(\bar{z}) \leq \varepsilon^{k_{s}} \|y^{k_{s}} - \bar{z}\|,
	\end{align*}
	which can be rewritten as
	\begin{align*}
	&g_0(y^{k_{s}}) - g_0(\bar{z}) + \sum_{i\in\ncvxeq}\lambda_i |g_{i}(y^{k_{s}})-g_{i}(\bar{z})| + \underset{i\in I_{ac}(\bar{z})}{\sum}\lambda_i \max\big(0,g_{i}(y^{k_{s}})-g_{i}(\bar{z})\big) \\
	&+ \underset{j\in J_{ac}(\bar{z})}{\sum}\tau_j \max\big(0,h_{j}(y^{k_{s}})-h_{j}(\bar{z})\big) \leq \varepsilon^{k_{s}} \|y^{k_{s}} - \bar{z}\|.
	\end{align*}
	Dividing both sides by $ \|y^{k_{s}} - \bar{z}\| $, we have
	\begin{align*}
	&\frac{g_0(y^{k_{s}})-g_0(\bar{z})}{\|y^{k_{s}} - \bar{z}\|} + \frac{\underset{i\in\ncvxeq}{\sum}\lambda_i |g_{i}(y^{k_{s}})-g_{i}(\bar{z})|}{\|y^{k_{s}} - \bar{z}\|} + \frac{\underset{i\in I_{ac}(\bar{z})}{\sum}\lambda_i \max\big(0,g_{i}(y^{k_{s}})-g_{i}(\bar{z})\big)}{\|y^{k_{s}} - \bar{z}\|} \\
	&+ \frac{\underset{j\in J_{ac}(\bar{z})}{\sum}\tau_j \max\big(0,h_{j}(y^{k_{s}})-h_{j}(\bar{z})\big)}{\|y^{k_{s}} - \bar{z}\|} \leq \varepsilon^{k_{s}}.
	\end{align*}
	Let $ k_s \rightarrow \infty $, then by \cref{lem:limit} and the definition of $ \xi $ from \cref{eq:xi}, we have
	\begin{align*}
	\nabla g_0(\bar{z})^T \xi &+ \sum_{i\in\ncvxeq} \lambda_i |\nabla g_{i}(\bar{z})^T \xi| + \underset{i\in I_{ac}(\bar{z})}{\sum} \lambda_i \max\big(0,\nabla g_{i}(\bar{z})^T \xi\big) \\
	&+ \underset{j\in J_{ac}(\bar{z})}{\sum} \tau_j \max\big(0,\nabla h_{j}(\bar{z})^T \xi\big) \leq 0.
	\end{align*}
	Now recall the KKT condition in \cref{thm:kkt}. We subtracting the product of \cref{eq:kkt} and $ \xi $ from the above equation, we have
	\begin{align*}
	&\sum_{i\in\ncvxeq} \Big[ \lambda_i |\nabla g_{i}(\bar{z})^T \xi| - \bar{\mu}_i \nabla g_i(\bar{z})^T \xi \Big] + \underset{i\in I_{ac}(\bar{z})}{\sum} \Big[ \lambda_i \max (0,\nabla g_{i}(\bar{z})^T \xi) - \bar{\mu}_i \nabla g_i(\bar{z})^T \xi \Big] \\
	&+ \underset{j\in J_{ac}(\bar{z})}{\sum} \Big[ \tau_j \max (0,\nabla h_{j}(\bar{z})^T \xi) - \bar{\sigma}_j \nabla h_j(\bar{z})^T \xi \Big] \leq 0,
	\end{align*}
	where $ \bar{\mu}_i $ and $ \bar{\sigma}_j $ are Lagrange multipliers associated with constraints.
	Due to the exactness property in \cref{thm:exactness}, these three terms are all nonnegative, and by the strict complementary slackness property in \cref{asup:SCS}, we have
	\begin{align*}
	&\nabla g_{i}(\bar{z})^T \xi = 0, \quad \forall \, i \in \ncvxeq \cup I_{ac}(\bar{z}), \\
	&\nabla h_{j}(\bar{z})^T \xi = 0, \quad \forall \, j \in J_{ac}(\bar{z}),
	\end{align*}
	and therefore,
	\begin{equation*}
	\big[ \nabla g_{i}(\bar{z})^T, \, \nabla h_{j}(\bar{z})^T \big]\, \xi = 0, \quad \forall \, i \in\ncvxeq \cup I_{ac}(\bar{z}), \, j \in J_{ac}(\bar{z}).
	\end{equation*}
	However, by~\cref{lem:basis} we know that the column space of $ [ \nabla g_{i}(\bar{z})^T, \, \nabla h_{j}(\bar{z})^T ] $ contains a basis of $ \real^\nz $. Since $ \xi \in \real^\nz $, this implies that $ \xi = 0 $, which contradicts the fact that $ \| \xi \| = 1 $, and thus \cref{eq:growth} holds true.
\end{proof}

\cref{lem:growth} provides an important condition that is satisfied by many optimal control problems. Next, we only need to show that given this condition, the \scvx procedure will indeed converge superlinearly. To proceed, let us first denote the stack of $ g_0(\cdot), g_i(\cdot) $, and $ h_j(\cdot) $ as $ G(\cdot) \in C^1 $, and represent $ J(\cdot) $ by a function composition $ \psi(G(\cdot)) $, and qualitatively we can write
\begin{equation*}
\psi(G(\cdot)) = G_{cost}(\cdot) + |G_{eq}(\cdot)| + \max\left( 0, G_{ineq}(\cdot)\right).
\end{equation*}
Note that $ \psi(\cdot) $ is convex since both $ |\cdot| $ and $ \max(0, \cdot) $ are convex functions.

\begin{lemma}
	Under the assumptions of~\cref{lem:growth}, there exists $ \gamma >0 $, such that
	\begin{equation} \label{eq:ineq_critical}
	\psi\left( G(\bar{z}) + \nabla G(\bar{z})^T \, d \right) \geq \psi(G(\bar{z})) + \gamma \| d \|, \quad \forall \, d \in \real^\nz.
	\end{equation}
\end{lemma}
\begin{proof}
	First we show that the statement is true for any small step $ d_\delta $ such that $ \|d_\delta\| \leq \delta $, where $ \delta $ is defined in~\cref{lem:growth}. We have
	\begin{align*}
	\psi\left( G(\bar{z}) + \nabla G(\bar{z})^T \, d_\delta \right) - \psi(G(\bar{z})) = \; &\psi(G(\bar{z}+d_\delta)) - \psi(G(\bar{z})) \\
	&+ \psi\left( G(\bar{z}) + \nabla G(\bar{z})^T \, d_\delta \right) - \psi(G(\bar{z}+d_\delta)) \\
	\geq \;&\beta \|d_\delta\| + \orderof(\|d_\delta\|) \\
	\geq \;&\frac{\beta}{2} \|d_\delta\|.
	\end{align*}
	Let $ \gamma \definedas \frac{\beta}{2} $, then we have \cref{eq:ineq_critical} hold for a small step $ d_\delta $.
	Note that the first inequality is due to~\cref{lem:growth} and again the fact that $ \orderof(\|d_\delta\|) $ can be taken out of $ |\cdot| $ and $ \max(0,\,\cdot\,) $ and $ G(\cdot) \in C^1 $.
	
	Now to generalize this result to any $ d \in \real^\nz, d \neq 0 $, we first define
	$$ \zeta \definedas \min\left( 1, \delta/\|d\| \right), $$
	and let $z_\zeta \definedas \bar{z} + \zeta d $. Denoting $ (z_\zeta - \bar{z}) $ as $ d_\zeta $, we have
	$ d_\zeta = \zeta \, d $.
	With this definition of $ \zeta $, one can verify that $ \|d_\zeta\|\leq \delta $. Hence,~\cref{eq:ineq_critical} holds true for $ d_\zeta $. Therefore, we have
	\begin{align*}
	\gamma \zeta \|d\| = \gamma \|d_\zeta\| &\leq \psi\left( G(\bar{z}) + \nabla G(\bar{z})^T \, d_\zeta \right) - \psi(G(\bar{z})) \\
	&= \psi\left( G(\bar{z}) + \zeta \nabla G(\bar{z})^T \, d \right) - \psi(G(\bar{z})) \\
	&= \psi\left( (1-\zeta)G(\bar{z}) + \zeta\left[ G(\bar{z}) + \nabla G(\bar{z})^T \, d \right] \right) - \psi(G(\bar{z})) \\
	&\leq (1-\zeta)\psi(G(\bar{z})) + \zeta \psi\left(G(\bar{z}) + \nabla G(\bar{z})^T \, d \right) - \psi(G(\bar{z})) \\
	&=\zeta \left[ \psi\left(G(\bar{z}) + \nabla G(\bar{z})^T \, d \right) - \psi(G(\bar{z})) \right].
	\end{align*}
	The second inequality is due to the convexity of $ \psi $ (note that $ 0 \leq \zeta \leq 1 $). Dividing both sides by $ \zeta $, we obtain~\cref{eq:ineq_critical}.
\end{proof}

\begin{theorem} [\textbf{Superlinear Convergence}] \label{thm:superlinear}
	Given a sequence $ \{z^k\} \rightarrow \bar{z} $ generated by the \scvx algorithm (\cref{algo:SCvx}) and suppose that \cref{asup:LICQ,asup:SCS,asup:active} are satisfied, then we have
	\begin{equation}
	\|z^{k+1} - \bar{z}\| = \orderof(\|z^{k} - \bar{z}\|).
	\end{equation}
\end{theorem}
\begin{proof}
	First consider the case without trust region constraints. Since $ d^k \definedas z^{k+1} - z^k $ is the unconstrained optimal solution to the convex subproblem, we have $ \forall \, d \in \real^\nz $,
	$$ \psi\big( G(z^k) + \nabla G(z^k)^T\,d \big) \geq \psi\big( G(z^k) + \nabla G(z^k)^T\,d^k \big). $$
	Let $ d = \bar{z} - z^k $. Then we have,
	\begin{equation*}
	\psi\left( G(z^k) + \nabla G(z^k)^T\,(\bar{z} - z^k) \right) - \psi\left( G(z^k) + \nabla G(z^k)^T\,(z^{k+1} - z^k) \right) \geq 0.
	\end{equation*}
	Together with~\cref{eq:ineq_critical}, we have
	\begin{align*}
	\gamma \|z^{k+1} - \bar{z}\| \leq \;&\psi\left( G(\bar{z}) + \nabla G(\bar{z})^T\,(z^{k+1} - \bar{z})  \right) - \psi\big(G(\bar{z})\big) \\
	&+ \psi\left( G(z^k) + \nabla G(z^k)^T\,(\bar{z} - z^k) \right) \\
	&- \psi\left( G(z^k) + \nabla G(z^k)^T\,(z^{k+1} - z^k) \right).
	\end{align*}
	Since $ \psi $ is convex and has a compact domain, it is (locally) Lipschitz continuous (see e.g. \cite{convex_lipschitz}). Therefore, we have
	\begin{equation*}
	\psi\left( G(z^k) + \nabla G(z^k)^T\,(\bar{z} - z^k) \right) - \psi\big(G(\bar{z})\big) \leq L \|G(z^k) - G(\bar{z}) + \nabla G(z^k)^T\,(\bar{z} - z^k) \|,
	\end{equation*}
	where $ L $ is the Lipschitz constant, and 
	\begin{align*}
	&\psi\left( G(\bar{z}) + \nabla G(\bar{z})^T\,(z^{k+1} - \bar{z})  \right) - \psi\left( G(z^k) + \nabla G(z^k)^T\,(z^{k+1} - z^k) \right) \\
	&\quad \leq L \|G(\bar{z}) - G(z^k) - \nabla G(z^k)^T\,(\bar{z} - z^k) + \left[ \nabla G(\bar{z})-\nabla G(z^k)\right]^T \left( z^{k+1} - \bar{z}\right) \| \\
	&\quad\leq L \|G(\bar{z}) - G(z^k) - \nabla G(z^k)^T\,(\bar{z} - z^k) \| + L \|\left[ \nabla G(\bar{z})-\nabla G(z^k) \right]^T \left( z^{k+1} - \bar{z}\right) \|.
	\end{align*}
	Combining the two parts, we obtain
	\begin{align*}
	\|z^{k+1} - \bar{z}\| \leq &\; \frac{2L}{\gamma} \|G(\bar{z}) - G(z^k) - \nabla G(z^k)^T\,(\bar{z} - z^k) \| \\
	& + \frac{L}{\gamma} \|\left[ \nabla G(\bar{z})-\nabla G(z^k) \right]^T \left( z^{k+1} - \bar{z}\right) \|.
	\end{align*}
	Since $ G(\cdot) \in C^1 $, we have
	$$ \|G(\bar{z}) - G(z^k) - \nabla G(z^k)^T\,(\bar{z} - z^k) \| = \orderof(\|z^{k} - \bar{z}\|). $$
	Given the fact that as $ k\rightarrow \infty$, $(\nabla G(\bar{z})-\nabla G(z^k))\rightarrow 0 $, we have
	$$ \|\left[ \nabla G(\bar{z})-\nabla G(z^k) \right]^T \left( z^{k+1} - \bar{z}\right) \| = \orderof(\|z^{k+1} - \bar{z}\|). $$
	Combining these two results, we obtain that
	$$ \|z^{k+1} - \bar{z}\| = \orderof(\|z^{k} - \bar{z}\|). $$
	Now, consider the \scvx algorithm with trust region constraints. From~\cref{eq:rho_goes_to_1} in the proof of \cref{lem:main}, $ \rho^k \rightarrow 1 $, we have that there exists $ k' $ large enough such that $ \rho^k > \rho_1 $ for all $ k \geq k' $, where $ \rho^k $ is the ratio defined in \cref{eq:ratio}. This means there must be some trust region radius $ r^{k'}>0 $ so that $ r^k \geq r^{k'} $ for all $ k \geq k' $, because we will not shrink the trust region radius from iteration $ k' $ on. On the other hand, from the global convergence result (\cref{thm:global_converge,thm:strong_converge}), we know that as $k$ increases, $ d^k \rightarrow 0 $. Therefore, there must exist a $ k' $, such that
	$$ \|d^k\|<r^{k'}, \quad \forall \, k \geq k', $$
	implying that the trust region constraints will eventually become inactive. Therefore, the above conclusion about the unconstrained problem holds for the trust region constrained case as well.
\end{proof}
With \cref{thm:superlinear}, we have in theory established superlinear convergence rate of the \scvx algorithm for many optimal control problems. Next we will have the claim verified through numerical simulations.

%===============================================================================
\section{Numerical Results} \label{sec:numerical_results}

In this section, we present a non-convex quad-rotor motion planning example problem to demonstrate the convergence rate of the \scvx algorithm. The problem was posed as a fixed-final-time minimum-fuel optimal control problem, and included non-convexities such as obstacle keep-out zones and nonlinear aerodynamic drag. See \cref{figure1} for a visual illustration of the problem setup.
\begin{figure}[!ht]
	\begin{centering}
		\includegraphics[width=0.9\textwidth]{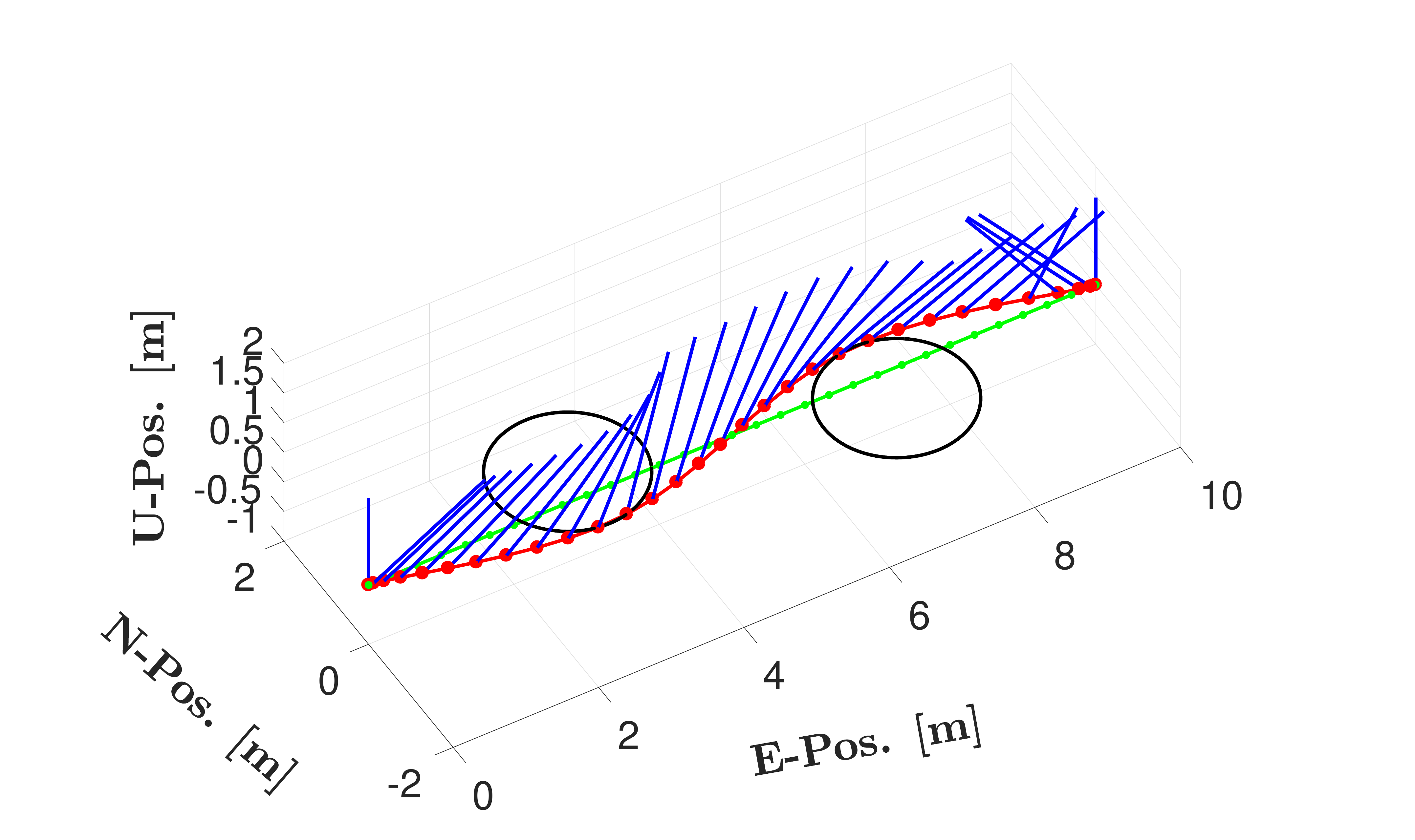}
		\captionsetup{width=0.9\textwidth}
		\caption{\emph{A perspective view of the converged trajectory. The obstacles are represented by the black circles. The red dots and blue lines represent the time discretized positions and thrust vectors, respectively. The green dots represent the initial guess, i.e., a straight line. The motion of the vehicle is from left to right.}}
		\label{figure1}
	\end{centering}
\end{figure}

Our algorithm was implemented in MATLAB using \texttt{CVX}~\cite{cvx} and \texttt{SDPT3}~\cite{sdpt3_pap}. For comparison, we also implemented the problem using SQP solver \texttt{SNOPT}~\cite{snopt} and general purpose IPM solver MATLAB's \texttt{fmincon} function. Specifically, for each method we compared the converged solutions, and the convergence rate attained in computing said solutions.

We intentionally did not compare the computation time of the algorithms, since we lacked sufficient insight into the internal implementation of the underlying algorithms (i.e. \texttt{CVX}, \texttt{SDPT3}, \texttt{SNOPT} and \texttt{fmincon}). However, on numerous occasions, we have implemented the \scvx algorithm using our in-house generic IPM solver \cite{dueri2014automated}, whose customized variant was used in recent autonomous rocket landing experiments \cite{dueri2016customized}. For example, we used the \scvx algorithm in conjunction with our in-house IPM to solve real-time onboard quad-rotor motion planning problems (similar to the example problem presented here) at rates greater than 8 Hz \cite{szmuk2017convexification}.

For tractability, we model the system using three degree-of-freedom translational dynamics, as in \cite{szmuk2017convexification}, and assume that feedback controllers provide sufficiently fast attitude tracking. Unless otherwise stated, assume that the notation defined in this section supersedes the notation defined in~\cref{sec:successive_convexification,sec:global}. The position, velocity, and thrust vectors of the vehicle at time $t$ are given by $\posvec(t)\in\real^3$, $\velvec(t)\in\real^3$, and $\bvec{T}(t)\in\real^3$, respectively. We use $\bvec{g}\in\real^3$ to denote the gravity vector. Additionally, $m\in\real_{++}$ and $k_D\in\real_+$ are used to denote the vehicle's mass and drag constant, respectively. In defining $k_D$, we combine the effects of air density, drag reference area, and coefficient of drag, and assume that all are constant. The continuous-time dynamics of the system are thus given by:
\begin{align*}
\dot{\posvec}(t) &= \velvec(t), \\
\dot{\velvec}(t) &= \frac{1}{m}\bvec{T}(t)-k_D\|\velvec(t)\|_2\velvec(t)+\bvec{g}.
\end{align*}
To implement this problem numerically, we discretize the fixed-final-time problem of duration $t_f$ into $\tf-1$ evenly spaced temporal intervals of duration $\Delta t$. To achieve this, the system is linearized about a trajectory, and then the corresponding discrete-time $A$- and $B$- matrices are computed accordingly. To linearize, we define the state and control vectors as
\begin{align*}
\bvec{x}(t)&\definedas\mat{\posvec^T(t),\thinspace\velvec^T(t)}^T, \\
\bvec{u}(t)&\definedas\bvec{T}(t).
\end{align*}
Using this notation, the original nonlinear dynamics are expressed as
$$\dot{\bvec{x}}(t) = f\big(\bvec{x}(t),\bvec{u}(t)\big), $$
and the linearized system can be written as
$$ \dot{\bvec{x}}(t) = A(t)\bvec{x}(t)+B(t)\bvec{u}(t)+\bvec{z}(t), $$
where
$$ \bvec{z}(t) \definedas f\big(\bvec{x}(t),\bvec{u}(t)\big)-A(t)\bvec{x}(t)-B(t)\bvec{u}(t). $$
For the first succession, we evaluate this linearization about an initialization trajectory, and for all subsequent successions we evaluate it about the previous iteration.

To complete the discretization process, we assume a first-order-hold (direct collocation) on $\bvec{u}(t)$ over $t\in[t_i,t_{i+1}]$ for each $i\in\mathcal{I}^-$, where $\mathcal{I}^- \definedas \{1,2,\ldots,\tf-1\}$. That is we define the control input $\bvec{u}(t)$ as
$$
\left.\begin{aligned}
&\bvec{u}(t) = \beta^{-}(t)\bvec{u}_i + \beta^{+}(t)\bvec{u}_{i+1} \\
&\beta^{-}(t) \definedas (t_{i+1}-t)/\Delta t \\
&\beta^{+}(t) \definedas (t-t_{i})/\Delta t
\end{aligned}\thinspace\right\}\thinspace\thinspace
\begin{aligned}
&t\in[t_i,t_{i+1}], \\
&\forall i\in\mathcal{I}^-.
\end{aligned}
$$
Noting that the state transition matrix (i.e. the discrete-time $A$-matrix) associated with the linearized system is governed by
$$ \dot{\Phi}_A(t) = A(t)\Phi_A(t,t_i), $$
we perform numerical integration to obtain the following matrices for all $i\in\mathcal{I}^-$:
\begin{align*}
A_{d,i} &\definedas \Phi_A(t_{i+1},t_i), \\
B_{d,i}^{-} &\definedas \int_{t_i}^{t_{i+1}}{\beta^{-}(\tau)\Phi_A(t_{i+1},\tau)B(\tau)d\tau}, \\
B_{d,i}^{+} &\definedas \int_{t_i}^{t_{i+1}}{\beta^{+}(\tau)\Phi_A(t_{i+1},\tau)B(\tau)d\tau}, \\
\bvec{z}_{d,i} &\definedas \int_{t_i}^{t_{i+1}}{\Phi_A(t_{i+1},\tau)\bvec{z}(\tau)d\tau}.
\end{align*}
We impose initial and final position, velocity, and thrust constraints, minimum and maximum thrust magnitude constraints, and thrust tilt constraints. The non-convex minimum thrust magnitude constraint can be imposed  by introducing the relaxation variable $\Gamma(t)\in\real$, and applying the lossless convexification technique proposed in \cite{behcet_aut11} and employed in \cite{szmuk2017convexification}. Additionally, to simplify the presentation of the results, we add constraints that artificially restrict the motion of the vehicle to the horizontal plane. However, we emphasize that the algorithms are solving a three-dimensional problem, not a two-dimensional one.

We assume that $n_{obs}$ cylindrical obstacles are present. Defining the sets
\begin{equation*}
\mathcal{I} \definedas \{1,2,\ldots,\tf\}, \quad \mathcal{J} \definedas \{1,2,\ldots,n_{obs}\},
\end{equation*}
the obstacle avoidance constraint for the $j^{th}$ obstacle at the $i^{th}$ time instance is given by
$$
\left.\begin{aligned}
&\|\Delta\posvec_{j,i}\|_2 \geq R_{obs,j} \\
&\Delta\posvec_{j,i} \definedas \posvec_i-\posvec_{obs,j}
\end{aligned}\thinspace\right\}\thinspace\thinspace
\begin{aligned}
\forall i&\in\mathcal{I}, \\
\forall j&\in\mathcal{J}.
\end{aligned}
$$
Since these constraints are non-convex, we linearize them about a trajectory $\bar{\posvec}(t)$, and define $\Delta\bar{\posvec}_{j,i}\definedas \bar{\posvec}_i-\posvec_{obs,j}$.

Additionally, for notational convenience, we define the concatenated solution vector
\begin{equation} \label{eq:def_X}
\bvec{X} \definedas \mat{\bvec{x}_0^T,\ldots,\bvec{x}_{\tf-1}^T,\bvec{u}_0^T,\ldots,\bvec{u}_{\tf-1}^T}^T,
\end{equation}
and use $r^0$ and $ r^k $ to denote the initial and current (i.e. at the $(k+1)^{th}$ iteration) size of the trust region, respectively. As outlined in~\cref{sec:global}, we enforce the trust region on the entire solution vector, $\bvec{X}$.%, and do so using the 1-norm, since it tends to be more computationally efficient for vectors of large dimensions.
\boxing{t}{
	\begin{problem}[h] \label{prob:main}
		: Convex Subproblem for the $(k+1)^{th}$ Succession of a Quad-Rotor Motion Planning Problem \\
		$$\underset{\bvec{u}}{\min}\thinspace\thinspace L \quad \textnormal{subject to:}$$
		$$L \definedas \sum_{i\in\mathcal{I}}{\Gamma_i\Delta t}+\lambda\sum_{i\in\mathcal{I}^-}{|\bvec{\nu}_i|}+\lambda\sum_{j\in\mathcal{J}}{\sum_{i\in\mathcal{I}}{\max\{0,\eta_{j,i}\}}}$$
		\begin{align*}
		\posvec_1 &= \posvec_{ic} & \posvec_{\tf} &= \posvec_{fc} \\
		\velvec_1 &= \velvec_{ic} & \velvec_{\tf} &= \velvec_{fc} \\
		\bvec{T}_1 &= \bvec{T}_{ic} & \bvec{T}_{\tf} &= \bvec{T}_{fc}
		\end{align*}
		$$ \bvec{x}_{i+1} = A_{d,i}\bvec{x}_i+B_{d,i}^{-}\bvec{u}_i+B_{d,i}^{+}\bvec{u}_{i+1}+\bvec{z}_{d,i}+\bvec{\nu}_i $$
		$$ R_j-\|\Delta\bar{\posvec}_{j,i}\|_2 - \frac{\Delta\bar{\posvec}_{j,i}}{\|\Delta\bar{\posvec}_{j,i}\|_2}\Delta\posvec_{j,i} \leq \eta_{j,i}$$
		\begin{align*}
		\mat{1\thinspace\thinspace 0 \thinspace\thinspace 0}\posvec_i &= 0 \\ 
		\|\bvec{T}_i\|_2&\leq\Gamma_i \\
		T_{min}&\leq\Gamma_i\leq T_{max} \\
		\cos\theta_{max}\Gamma_i &\leq \mat{1 \thinspace\thinspace 0 \thinspace\thinspace 0}\bvec{T}_i 
		\end{align*}
		$$ \|\bvec{X}\|_1 \leq r^k$$
	\end{problem}
}
A summary of the convex subproblem solved at each succession is provided in~\cref{prob:main}. \cref{table1} provides the \scvx algorithm parameters, the problem parameters, and the boundary conditions used to obtain our results. Note that the parameters and results are presented assuming an Up-East-North reference frame.

%For the purpose of comparison, the same problem was also implemented in MATLAB's \texttt{fmincon} using a combination of linear and nonlinear equality and inequality constraints.

\begin{table}[h]
	\caption{\scvx Parameters, Problem Parameters, and B.C.'s}
	\label{table1}
	\begin{center}
		\begin{tabular}{lll}
			\hhline{===}
			Parameter & Value & Units \\
			\hline
			$\lambda$ \hspace{2cm} & $1e^5$ \hspace{2cm} & - \hspace{2cm} \\
			$r^0$ & $1.0$ & - \\
			$\alpha$ & $2.0$ & - \\
			$\beta$ & $3.2$ & - \\
			$\epsilon_{tol}$ & $1e^{-3}$ & - \\
			$\rho_0$ & $0.00$ & - \\
			$\rho_1$ & $0.25$ & - \\
			$\rho_2$ & $0.7$ & - \\
			%			\hdashline[0.5pt/1.5pt]
			$\tf$ & $31$ & - \\
			$t_f$ & 3.0 & [s] \\
			$m$ & 0.3 & [kg] \\
			$T_{min}$ & $1.0$ & [N] \\
			$T_{max}$ & $4.0$ & [N] \\
			$\theta_{max}$ & 45 & [$\degree$] \\
			$k_D$ & $0.5$ & - \\
			$\bvec{g}$ & $[-9.81 \thinspace\thinspace 0 \thinspace\thinspace 0]^T$ & [m/s$^2$] \\
			$n_{obs}$ & $2$ & - \\
			$R_{obs,1} $ & $1$ & [m] \\
			$R_{obs,2} $ & $1$ & [m] \\
			$\posvec_{obs,1} $ & $\mat{0 \thinspace\thinspace 3 \thinspace\thinspace 0.45}^T$ & [m] \\
			$\posvec_{obs,2} $ & $\mat{0 \thinspace\thinspace 7 \thinspace -0.45}^T$ & [m] \\
			%			\hdashline[0.5pt/1.5pt]
			$\posvec_{ic}$ & $[0\thinspace\thinspace 0 \thinspace\thinspace 0]^T$ & [m] \\
			$\velvec_{ic}$ & $[0\thinspace\thinspace 0.5 \thinspace\thinspace 0]^T$ & [m/s] \\
			$\bvec{T}_{ic}$ & $-mg$ & [N] \\
			$\posvec_{fc}$ & $[0\thinspace\thinspace 10 \thinspace\thinspace 0]^T$ & [m] \\
			$\velvec_{fc}$ & $[0\thinspace\thinspace 0.5 \thinspace\thinspace 0]^T$ & [m/s] \\
			$\bvec{T}_{fc}$ & $-mg$ & [N] \\
			\hhline{===}
		\end{tabular}
	\end{center}
\end{table}
Aside from the problem setup, \cref{figure1} also shows the trajectory corresponding to the converged solution. The obstacles keep-out zones are shown as the black circles, and the optimized path is shown in red. The red dots indicate each of the $\tf$ discretization points, and the blue lines represent the thrust vectors at said time points. The large tilt angles generated by the optimization are necessary to counter the drag force dictated by our choice of $k_D$.

\cref{figure2} shows the positions and velocities of the converged trajectory. The transit time for the trajectory is largely dictated by the east velocity of the vehicle, which in turn is dictated by the prescribed travel distance of 10 [m], and prescribed final time of 3.0 [s].

\cref{figure3} shows the corresponding thrust magnitude and thrust tilt angle. As seen in the figure, the thrust magnitude rides the maximum thrust constraint through most of the trajectory. The constant altitude constraint thus translates to a tilt angle that provides enough vertical thrust to balance out the force of gravity. The trajectory terminates with a quick reversal in thrust direction to bring the vehicle to the prescribed final velocity.

\cref{figure4} shows the differences in position, velocity, and thrust between the \scvx and SQP solutions. Evidently, the solutions are effectively identical.

\cref{figure5} shows a comparison between the convergence rates of the \scvx and SQP algorithms, using the definition of $\bvec{X}$ given in~\cref{eq:def_X}. The same initial guess was used to initialize both methods. This result shows that the \scvx algorithm initially converged at a slower rate than that of the SQP algorithm, but that after a few iterations, the \scvx algorithm attains a far faster convergence rate, which is the typical fashion of superlinear convergence. Specifically, the \scvx algorithm converged to a tolerance of $\Delta L_{tol}$ (see~\cref{table1}) in 11 iterations, whereas the SQP algorithm obtained about the same level of convergence in approximately 16 iterations. It is worthwhile mentioning that the first few iterations of the \scvx convergence process were spent decreasing the virtual control, $\bvec{\nu}_i$, and virtual buffer zone, $\eta_{j,i}$, terms (see~\cref{prob:main}), effectively avoiding artificial infeasibility, and recovering physical feasibility. The remaining iterations refined the physically feasible trajectory to the specified tolerance.

Lastly, \cref{figure6} also compares the convergence rates with the addition of IPM. The same initial guess was used to initialize all three methods. This result shows that general purpose IPM, at least the \texttt{fmincon}'s implementation, performs poorly in solving this problem. It struggles to find a descent direction, and takes a large number of iterations (around 75) to achieve the same level of accuracy.
\begin{figure}[!ht]
	\begin{centering}
		\captionsetup{width=0.8\textwidth}
		\includegraphics[width=0.8\textwidth]{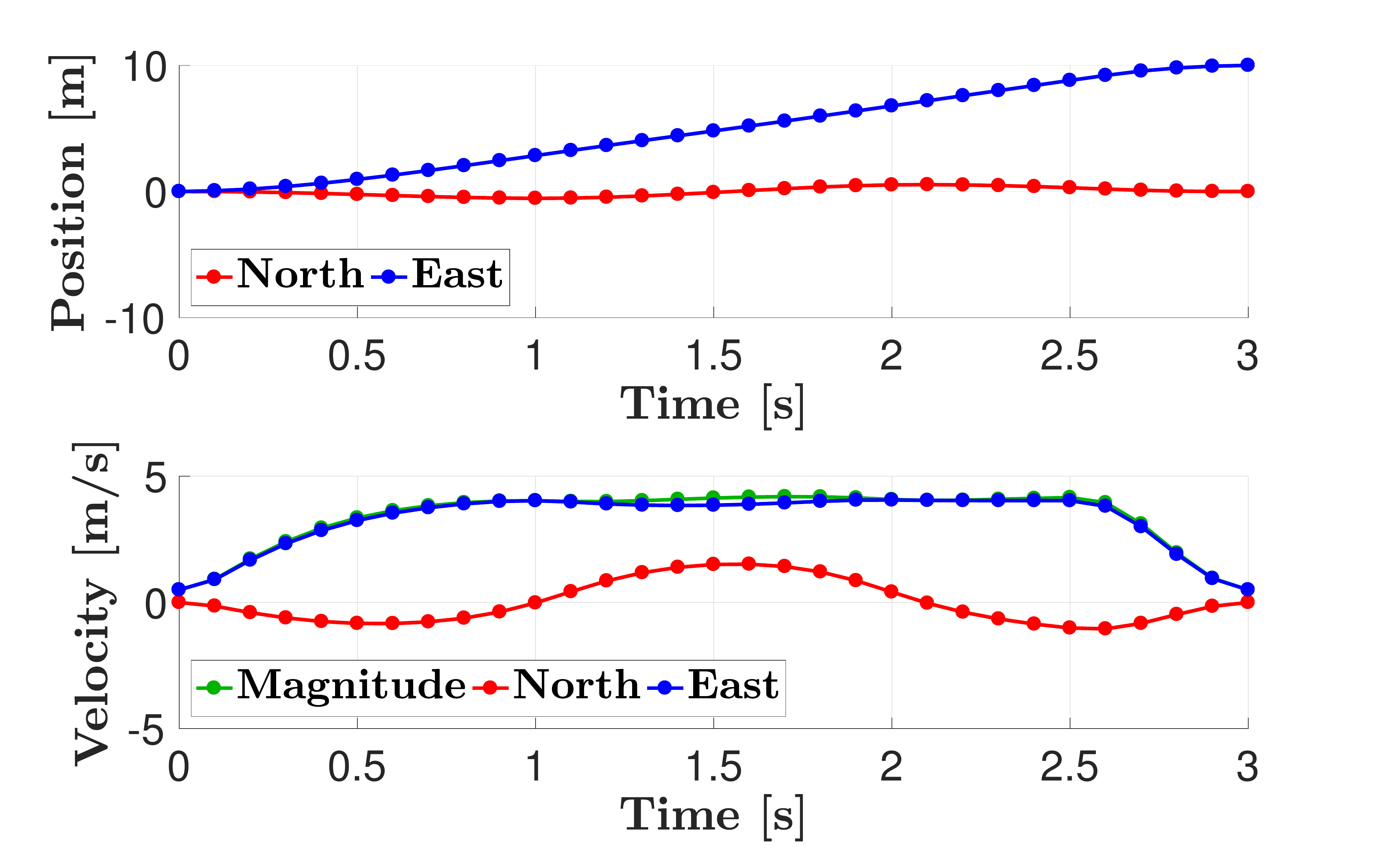}
		\caption{\emph{Positions and velocities of the converged trajectory.}}
		\label{figure2}
		%\vspace{0.25cm}
		\includegraphics[width=0.8\textwidth]{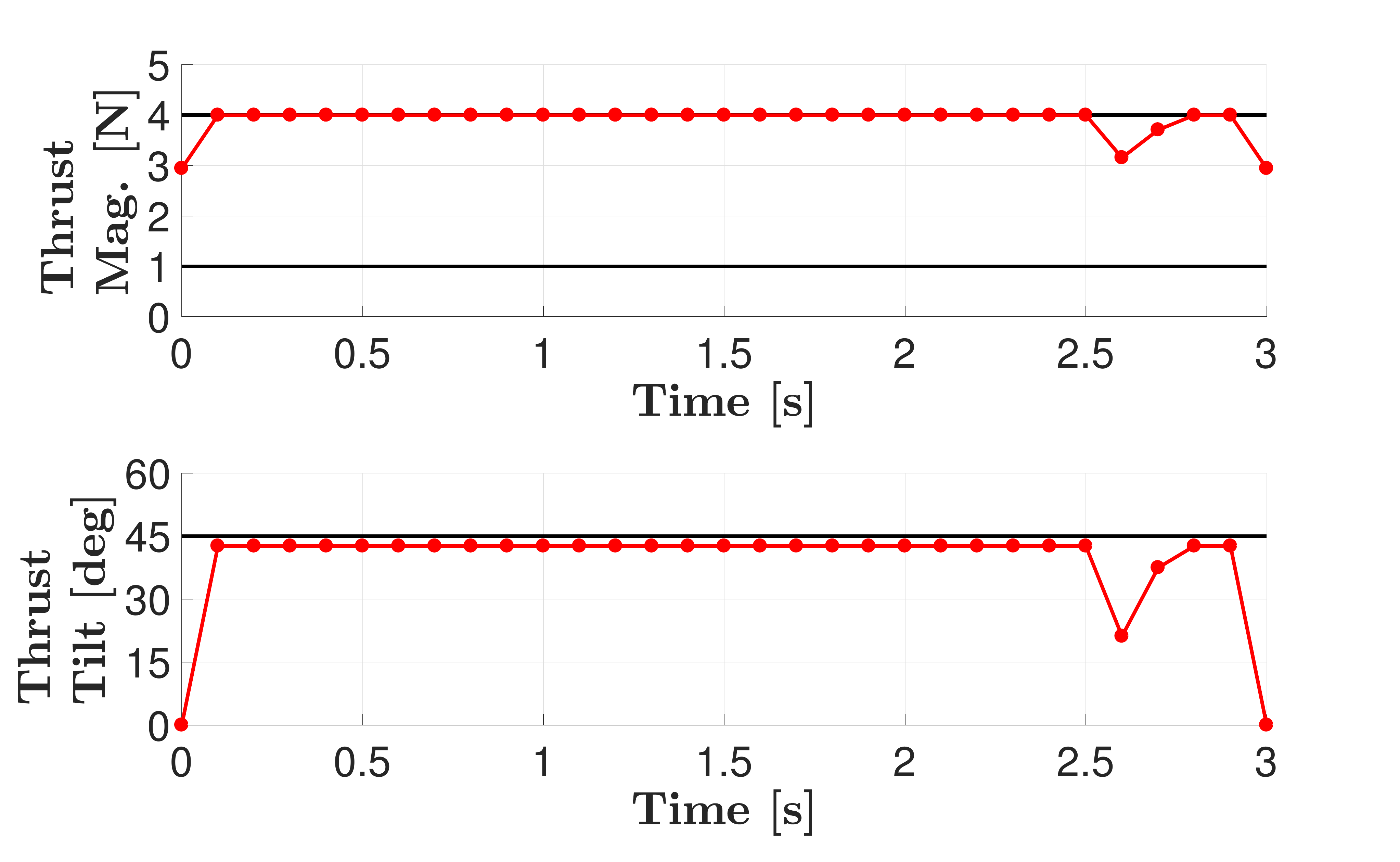}
		\caption{\emph{Thrust magnitude and tilt versus time for the converged trajectory. Black lines show the minimum and maximum thrust bounds in the top plot, and the maximum thrust tilt angle in the bottom plot.}}
		\label{figure3}
		%\vspace{0.25cm}
		\includegraphics[width=0.8\textwidth]{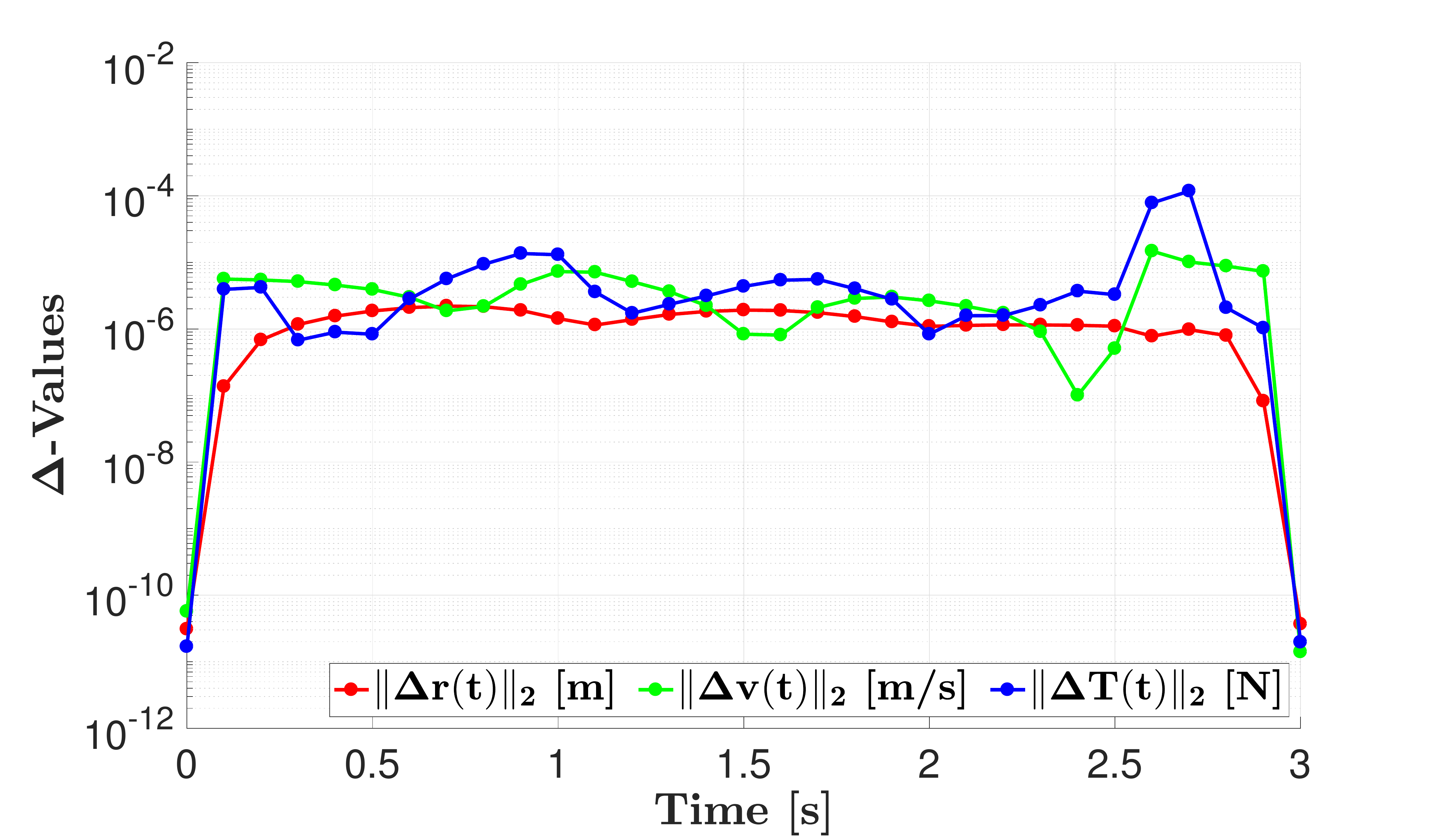}
		\caption{\emph{Position, velocity, and thrust differences between the converged \scvx solution and the converged SQP solution.}}
		\label{figure4}
	\end{centering}
\end{figure}
\begin{figure}[!ht]
	\begin{centering}
		\captionsetup{width=0.8\textwidth}
		\includegraphics[width=0.8\textwidth]{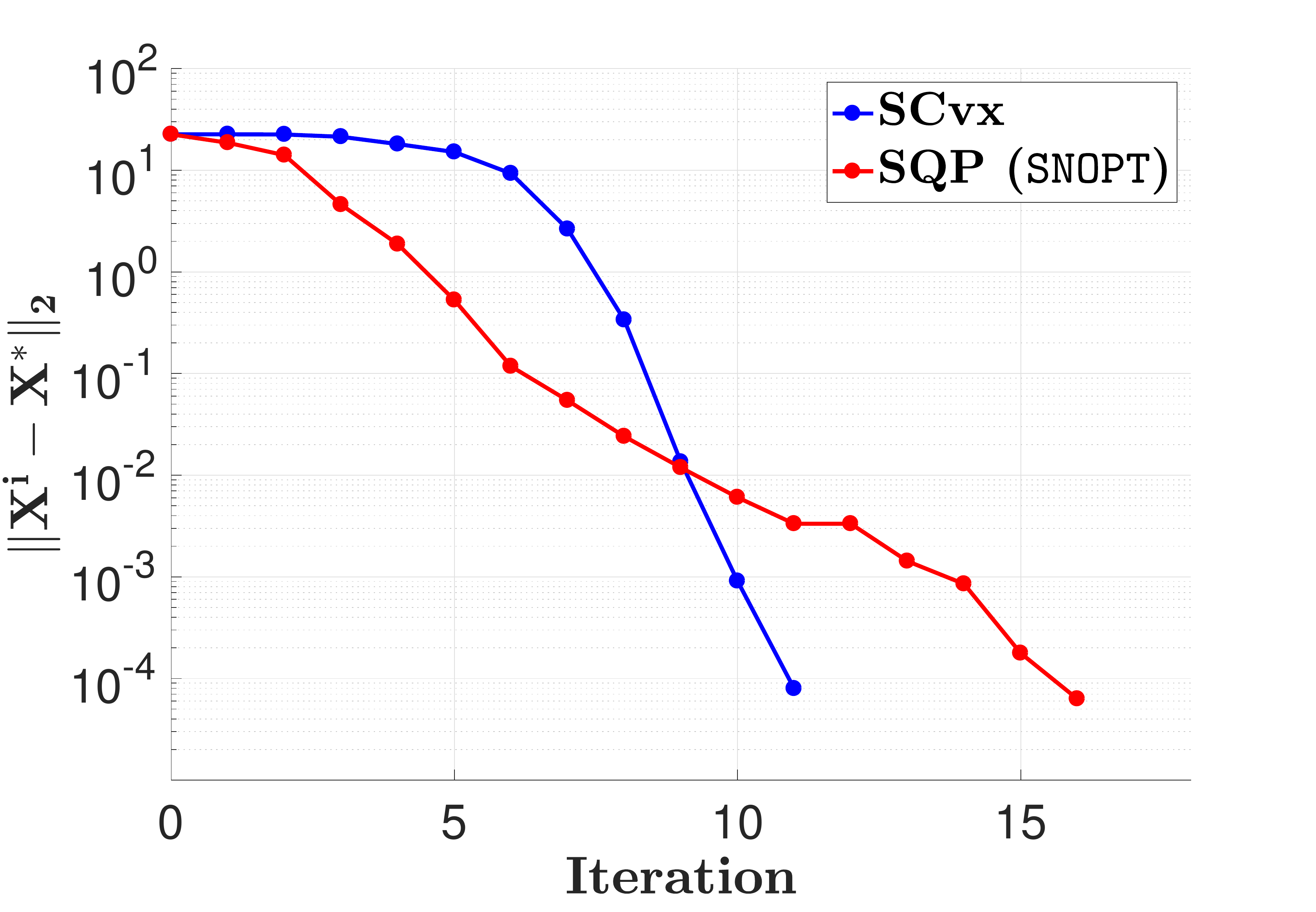}
		\caption{\emph{Convergence history of the \scvx (blue) and SQP (red) algorithms. The lines show the magnitude of the difference between $\bvec{X}$ at each iteration and the converged solution, $\bvec{X}^*$, and their slope indicates the rate of convergence.}}
		\label{figure5}
		%\vspace{0.25cm}
		\includegraphics[width=0.8\textwidth]{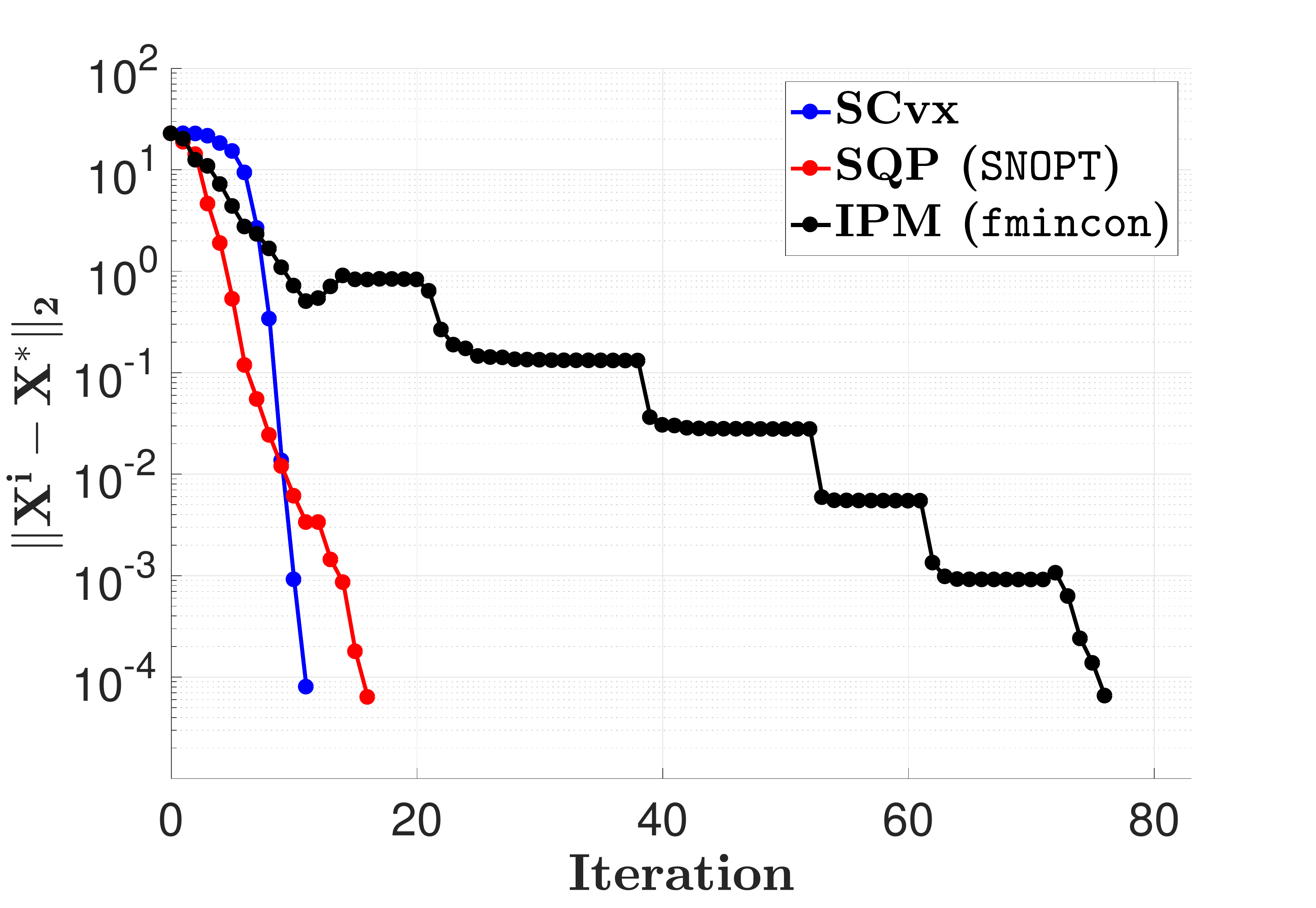}
		\caption{\emph{Convergence history of the \scvx (red), SQP (blue) and IPM (black) algorithms. The lines show the magnitude of the difference between $\bvec{X}$ at each iteration and the converged solution, $\bvec{X}^*$, and their slope indicates the rate of convergence.}}
		\label{figure6}
	\end{centering}
\end{figure}

%===============================================================================
\section{Conclusion} \label{sec:conclusion}

In this paper, we have proposed an enhanced version of the \scvx algorithm, which is able to solve optimal control problems with nonlinear dynamics and non-convex state and control constraints. The algorithm computes solutions by solving a sequence of convex optimization subproblems, each obtained by linearizing the non-convexities around the solution of the previous iteration. Each subproblem employs virtual controls, virtual buffer zones, and a trust region as safe-guarding mechanisms against artificial infeasibility and artificial unboundedness. Although the linearization acts as an approximation during the convergence process, the converged solution solves the original non-convex optimal control problem exactly and with local optimality.

Further, to set the \scvx algorithm apart from existing SCP-like methods, we have given a thorough analysis of the convergence properties of the \scvx algorithm, providing proofs of both global convergence (weak and strong) and superlinear convergence rate. These theoretical results were further validated by numerical simulations of a non-convex quad-rotor motion planning example problem. Notably, our simulation results showed that while the \scvx algorithm made slower progress early in the convergence process, it ultimately converged to the specified tolerance in less (and in some cases much less) number of iterations than competing solvers.

In summary, the main technical result of this paper is twofold: (1) limit point of the \scvx algorithm is guaranteed to be a local optimum of the original non-convex optimal control problem (assuming it contains zero virtual control and virtual buffer zone terms), and (2) the \scvx algorithm converges superlinearly. Since the solution process is distilled into a sequence of subproblems that are guaranteed to be convex, the \scvx algorithm is well suited for real-time autonomous applications, as demonstrated in \cite{szmuk2017convexification,szmuk2018real}.

\section*{Acknowledgments}
The authors gratefully acknowledge John Hauser of University of Colorado for his valuable insights.

\bibliographystyle{siamplain}
\bibliography{scvx_ref}

\begin{thebibliography}{10}

\bibitem{absil2005convergence}
{\sc P.~Absil, R.~Mahony, and B.~Andrews}, {\em Convergence of the iterates of
  descent methods for analytic cost functions}, SIAM Journal on Optimization,
  16 (2005), pp.~531--547, \url{https://doi.org/10.1137/040605266}.

\bibitem{behcet_aut11}
{\sc B.~{A\c c\i kme\c se} and L.~Blackmore}, {\em Lossless convexification of
  a class of optimal control problems with non-convex control constraints},
  Automatica, 47 (2011), pp.~341--347.

\bibitem{pointing2013}
{\sc B.~{A\c c\i kme\c se}, J.~Carson, and L.~Blackmore}, {\em Lossless
  convexification of non-convex control bound and pointing constraints of the
  soft landing optimal control problem}, IEEE Transactions on Control Systems
  Technology, 21 (2013), pp.~2104--2113.

\bibitem{attouch2013convergence}
{\sc H.~Attouch, J.~Bolte, and B.~F. Svaiter}, {\em Convergence of descent
  methods for semi-algebraic and tame problems: proximal algorithms,
  forward--backward splitting, and regularized gauss--seidel methods},
  Mathematical Programming, 137 (2013), pp.~91--129.

\bibitem{augugliaro2012generation}
{\sc F.~Augugliaro, A.~P. Schoellig, and R.~D'Andrea}, {\em Generation of
  collision-free trajectories for a quadrocopter fleet: A sequential convex
  programming approach}, in International Conference on Intelligent Robots and
  Systems (IROS), IEEE, 2012, pp.~1917--1922.

\bibitem{berkovitz1974optimal}
{\sc L.~D. Berkovitz}, {\em Optimal control theory}, Springer-Verlag, 1974.

\bibitem{betts1998survey}
{\sc J.~T. Betts}, {\em Survey of numerical methods for trajectory
  optimization}, Journal of guidance, control, and dynamics, 21 (1998),
  pp.~193--207.

\bibitem{lars2016autonomous}
{\sc L.~Blackmore}, {\em Autonomous precision landing of space rockets}, The
  Bridge on Frontiers of Engineering, 4 (2016), pp.~15--20.

\bibitem{boggs1995sequential}
{\sc P.~T. Boggs and J.~W. Tolle}, {\em Sequential quadratic programming}, Acta
  numerica, 4 (1995), pp.~1--51.

\bibitem{bolte2007lojasiewicz}
{\sc J.~Bolte, A.~Daniilidis, and A.~Lewis}, {\em The {L}ojasiewicz inequality
  for nonsmooth subanalytic functions with applications to subgradient
  dynamical systems}, SIAM Journal on Optimization, 17 (2007), pp.~1205--1223.

\bibitem{clarke1975generalized}
{\sc F.~H. Clarke}, {\em Generalized gradients and applications}, Transactions
  of the American Mathematical Society, 205 (1975), pp.~247--262.

\bibitem{clarke1976new}
{\sc F.~H. Clarke}, {\em A new approach to lagrange multipliers}, Mathematics
  of Operations Research, 1 (1976), pp.~165--174.

\bibitem{conn2000trust}
{\sc A.~R. Conn, N.~I. Gould, and P.~L. Toint}, {\em Trust region methods},
  vol.~1, {SIAM}, 2000.

\bibitem{dall2013distributed}
{\sc E.~Dall'Anese, H.~Zhu, and G.~B. Giannakis}, {\em Distributed optimal
  power flow for smart microgrids}, IEEE Transactions on Smart Grid, 4 (2013),
  pp.~1464--1475, \url{https://doi.org/10.1109/TSG.2013.2248175}.

\bibitem{domahidi2013ecos}
{\sc A.~Domahidi, E.~Chu, and S.~Boyd}, {\em {ECOS}: An{ SOCP} solver for
  embedded systems}, in European Control Conference (ECC), IEEE, 2013,
  pp.~3071--3076.

\bibitem{dueri2016customized}
{\sc D.~Dueri, B.~A{\c{c}}{\i}kme{\c{s}}e, D.~P. Scharf, and M.~W. Harris},
  {\em Customized real-time interior-point methods for onboard powered-descent
  guidance}, Journal of Guidance, Control, and Dynamics, 40 (2016),
  pp.~197--212.

\bibitem{dueri2014automated}
{\sc D.~Dueri, J.~Zhang, and B.~A{\c{c}}ikmese}, {\em Automated custom code
  generation for embedded, real-time second order cone programming}, in 19th
  IFAC World Congress, 2014, pp.~1605--1612.

\bibitem{fletcher1987practical}
{\sc R.~Fletcher}, {\em Practical methods of optimization, 2nd Edition},
  vol.~2, Wiley, 1987.

\bibitem{garg2010unified}
{\sc D.~Garg, M.~Patterson, W.~W. Hager, A.~V. Rao, D.~A. Benson, and G.~T.
  Huntington}, {\em A unified framework for the numerical solution of optimal
  control problems using pseudospectral methods}, Automatica, 46 (2010),
  pp.~1843--1851.

\bibitem{snopt}
{\sc P.~E. Gill, W.~Murray, and M.~A. Saunders}, {\em {SNOPT}: An {SQP}
  algorithm for large-scale constrained optimization}, SIAM Rev., 47 (2005),
  pp.~99--131.

\bibitem{goldman1956theory}
{\sc A.~J. Goldman and A.~W. Tucker}, {\em Theory of linear programming},
  Linear inequalities and related systems, 38 (1956), pp.~53--97.

\bibitem{cvx}
{\sc M.~Grant and S.~Boyd}, {\em {CVX}: Matlab software for disciplined convex
  programming, version 2.1}.
\newblock \url{http://cvxr.com/cvx}, Mar. 2014.

\bibitem{hager2018convergence}
{\sc W.~W. Hager, J.~Liu, S.~Mohapatra, A.~V. Rao, and X.-S. Wang}, {\em
  Convergence rate for a gauss collocation method applied to constrained
  optimal control}, SIAM Journal on Control and Optimization, 56 (2018),
  pp.~1386--1411, \url{https://doi.org/10.1137/16M1096761}.

\bibitem{han1979exact}
{\sc S.~P. Han and O.~L. Mangasarian}, {\em Exact penalty functions in
  nonlinear programming}, Mathematical programming, 17 (1979), pp.~251--269.

\bibitem{harris2014maximum}
{\sc M.~W. Harris and B.~A{\c{c}}{\i}kme{\c{s}}e}, {\em Maximum divert for
  planetary landing using convex optimization}, Journal of Optimization Theory
  and Applications, 162 (2014), pp.~975--995.

\bibitem{harris2014minimum}
{\sc M.~W. Harris and B.~A{\c{c}}{\i}kme{\c{s}}e}, {\em Minimum time rendezvous
  of multiple spacecraft using differential drag}, Journal of Guidance,
  Control, and Dynamics, 37 (2014), pp.~365--373.

\bibitem{acado}
{\sc B.~Houska, H.~Ferreau, and M.~Diehl}, {\em {ACADO} {T}oolkit -- {A}n
  {O}pen {S}ource {F}ramework for {A}utomatic {C}ontrol and {D}ynamic
  {O}ptimization}, Optimal Control Applications and Methods, 32 (2011),
  pp.~298--312.

\bibitem{hull1997}
{\sc D.~Hull}, {\em Conversion of optimal control problems into parameter
  optimization problems}, Journal of Guidance, Control, and Dynamics, 20
  (1997), pp.~57--60.

\bibitem{kurdyka1998gradients}
{\sc K.~Kurdyka}, {\em On gradients of functions definable in o-minimal
  structures}, in Annales de l'institut Fourier, vol.~48, Chartres: L'Institut,
  1950-, 1998, pp.~769--784.

\bibitem{lasalle1953study}
{\sc J.~P. LaSalle}, {\em Study of the basic principles underlying the
  bang-bang servo}, Tech. Rep. GER-5518,  (1953).

\bibitem{liberzon2012calculus}
{\sc D.~Liberzon}, {\em Calculus of variations and optimal control theory: a
  concise introduction}, Princeton University Press, 2012.

\bibitem{liu2014solving}
{\sc X.~Liu and P.~Lu}, {\em Solving nonconvex optimal control problems by
  convex optimization}, Journal of Guidance, Control, and Dynamics, 37 (2014),
  pp.~750--765.

\bibitem{lojasiewicz1982trajectoires}
{\sc S.~Lojasiewicz}, {\em Sur les trajectoires du gradient d’une fonction
  analytique}, Seminari di geometria, 1983 (1982), pp.~115--117.

\bibitem{loxton2009optimal}
{\sc R.~C. Loxton, K.~L. Teo, V.~Rehbock, and K.~F.~C. Yiu}, {\em Optimal
  control problems with a continuous inequality constraint on the state and the
  control}, Automatica, 45 (2009), pp.~2250--2257.

\bibitem{mao2019mpc}
{\sc Y.~Mao, D.~Dueri, M.~Szmuk, and B.~A{\c{c}}{\i}kme{\c{s}}e}, {\em
  Convexification and Real-Time Optimization for MPC with Aerospace
  Applications}, Springer International Publishing, Cham, 2019, pp.~335--358,
  \url{https://doi.org/10.1007/978-3-319-77489-3_15}.

\bibitem{mao2017successive}
{\sc Y.~Mao, D.~Dueri, M.~Szmuk, and B.~Açıkmeşe}, {\em Successive
  convexification of non-convex optimal control problems with state
  constraints}, IFAC-PapersOnLine, 50 (2017), pp.~4063--4069,
  \url{https://doi.org/https://doi.org/10.1016/j.ifacol.2017.08.789}.

\bibitem{SCvx_cdc16}
{\sc Y.~Mao, M.~Szmuk, and B.~{A\c c\i kme\c se}}, {\em Successive
  convexification of non-convex optimal control problems and its convergence
  properties}, in 2016 IEEE 55th Conference on Decision and Control (CDC),
  2016, pp.~3636--3641.

\bibitem{mattingley2012}
{\sc J.~Mattingley and S.~Boyd}, {\em Cvxgen: A code generator for embedded
  convex optimization}, Optimization and Engineering, 13 (2012), pp.~1--27.

\bibitem{nesterov1994interior}
{\sc Y.~Nesterov and A.~Nemirovskii}, {\em Interior-Point Polynomial Algorithms
  in Convex Programming}, Society for Industrial and Applied Mathematics, 1994.

\bibitem{pontryagin1987mathematical}
{\sc L.~S. Pontryagin}, {\em Mathematical theory of optimal processes}, CRC
  Press, 1987.

\bibitem{rockafellar1970convex}
{\sc R.~T. Rockafellar}, {\em Convex Analysis}, vol.~28, Princeton University
  Press, 1970.

\bibitem{convex_lipschitz}
{\sc W.~S. U. M. D.~C. Room}, {\em Every convex function is locally lipschitz},
  The American Mathematical Monthly, 79 (1972), pp.~1121--1124,
  \url{http://www.jstor.org/stable/2317434}.

\bibitem{rudin1964principles}
{\sc W.~Rudin}, {\em Principles of mathematical analysis}, vol.~3, McGraw-Hill,
  1964.

\bibitem{schulman2014motion}
{\sc J.~Schulman, Y.~Duan, J.~Ho, A.~Lee, I.~Awwal, H.~Bradlow, J.~Pan,
  S.~Patil, K.~Goldberg, and P.~Abbeel}, {\em Motion planning with sequential
  convex optimization and convex collision checking}, The International Journal
  of Robotics Research, 33 (2014), pp.~1251--1270.

\bibitem{sussmann1983lie}
{\sc H.~J. Sussmann}, {\em Lie brackets, real analyticity and geometric
  control}, Differential geometric control theory, 27 (1983), pp.~1--116.

\bibitem{szmuk2018successive}
{\sc M.~Szmuk and B.~{A\c c\i kme\c se}}, {\em Successive convexification for
  6-dof mars rocket powered landing with free-final-time}, ArXiv e-prints,
  (2018).
\newblock arXiv:1802.03827.

\bibitem{szmuk2018real}
{\sc M.~{Szmuk}, C.~A. {Pascucci}, and B.~{AÇikmeşe}}, {\em Real-time
  quad-rotor path planning for mobile obstacle avoidance using convex
  optimization}, in 2018 IEEE/RSJ International Conference on Intelligent
  Robots and Systems (IROS), Oct 2018, pp.~1--9,
  \url{https://doi.org/10.1109/IROS.2018.8594351}.

\bibitem{szmuk2017convexification}
{\sc M.~Szmuk, C.~A. Pascucci, D.~Dueri, and B.~{A\c c\i kme\c se}}, {\em
  Convexification and real-time on-board optimization for agile quad-rotor
  maneuvering and obstacle avoidance}, in 2017 IEEE/RSJ International
  Conference on Intelligent Robots and Systems (IROS), Sep 2017,
  pp.~4862--4868.

\bibitem{sdpt3_pap}
{\sc K.~Toh, M.~Todd, and R.~Tutuncu}, {\em {SDPT3} --- a {M}atlab software
  package for semidefinite programming}, Optimization Methods and Software, 11
  (1999), pp.~545--581.

\bibitem{ipopt}
{\sc A.~W{\"a}chter and L.~T. Biegler}, {\em On the implementation of an
  interior-point filter line-search algorithm for large-scale nonlinear
  programming}, Mathematical Programming, 106 (2006), pp.~25--57,
  \url{https://doi.org/10.1007/s10107-004-0559-y}.

\end{thebibliography}
\end{document}